\documentclass[a4paper,12pt,oneside,openany]{article}

\usepackage{t1enc}

\usepackage[latin2]{inputenc}
\usepackage{amssymb}
\usepackage{amsmath,amsthm}
\usepackage{color}
\usepackage{colordvi}
\usepackage{mathrsfs}
\usepackage{subcaption}
\usepackage{graphicx}
\usepackage{latexsym}
\usepackage{epsfig}
\usepackage[a4paper,left=2cm,right=2cm,top=2cm,bottom=2cm]{geometry}

\numberwithin{equation}{section}
\newcounter{dummy}
\numberwithin{dummy}{section}

\newtheorem{theorem}[dummy]{Theorem}

\newtheorem{corollary}[dummy]{Corollary}
\newtheorem{lemma}[dummy]{Lemma}

\newtheorem{assumption}[dummy]{Assumption}

\newcommand{\Var}{\mathop{\rm Var}\nolimits}

\newcommand{\eps}{\varepsilon}

\renewcommand{\epsilon}{\varepsilon}

\begin{document}

\title{Optimal recovery and uncertainty quantification for distributed Gaussian process regression}

\author{
Amine Hadji\\[.5ex]
{\em Mathematical Institute, Leiden University}
\and 
Tammo Hesslink\\[.5ex]
{\em Department of Mathematics, University of Amsterdam}
\and
Botond Szab\'o\footnote{This project has received funding from the European Research Council (ERC) under the
European Union's Horizon 2020 research and innovation programme (grant agreement No.
101041064).}\\[.5ex]
{\em Department of Decision Sciences, Bocconi University}\\ {\em Bocconi Institute for Data Science and Analytics}
}
\date{}

\maketitle

\abstract{Gaussian Processes (GP) are widely used for probabilistic modeling and inference for nonparametric regression. However, their computational complexity scales cubicly with the sample size rendering them unfeasible for large data sets. To speed up the computations various distributed methods were proposed in the literature. These methods have, however, limited theoretical underpinning. In our work we derive frequentist theoretical guarantees and limitations for a range of distributed methods for general GP priors in context of the nonparametric regression model, both for recovery and uncertainty quantification. As specific examples we consider covariance kernels both with polynomially and exponentially decaying  eigenvalues. We demonstrate the practical performance of the investigated approaches in a numerical study using synthetic data sets.}

\section{Introduction}\label{sec: intr}
Gaussian processes are highly popular in statistical and machine learning. They are widely used for probabilistic modeling and inference, taken as priors for functional parameters in the Bayesian approach, see for instance the monographs \cite{rasmussen:williams:2006,kocijan2016modelling, ghosal2017fundamentals}.

Gaussian processes are stochastic processes whose finite-dimensional distributions are multivariate Gaussian vectors. They are specified by their mean and covariance function. In the Bayesian analysis one typically considers mean-zero Gaussian processes, whereas the covariance function can be chosen freely, which influence directly the shape of the Gaussian process prior. They are often used in context of the nonparametric regression model, nonparametric classification,... etc. In our analysis we focus on the former one, i.e. we consider $(X_i,Y_i)$, $i=1,...,N$, i.i.d pairs of random variables satisfying 
\begin{align}\label{eq: nonpara reg}
Y_i=f(X_i)+\varepsilon_i,~\varepsilon_i\stackrel{iid}{\sim}\mathcal{N}(0,\sigma^2),
\end{align}
with design points $X_i$, $i=1,...,N$, belonging to some compact set $\mathcal{X}\subset\mathbb{R}^d$, observations $Y_i\in\mathbb{R}$, noise variance $\sigma^2>0$ and functional parameter $f$. In the Gaussian process regression $f$ is modeled with a mean-zero Gaussian Process (GP) prior $GP(0,K)$, where $K:\mathbb{R}^d\times\mathbb{R}^d\mapsto\mathbb{R}$ is a positive definite kernel. The popularity of Gaussian Process regression stems from the explicit formula of the corresponding posterior distribution due to conjugacy.  Furthermore, since the outcome of a Bayesian procedure is not only a single estimate of the regression function but a whole distribution, the procedure provides built-in uncertainty quantification. Among the numerous fields using Gaussian process regression, one can find computer experiment emulations \cite{currin:mitchell:morris:ylvisaker:1991,kennedy:ohagan:2001,mitchell:sacks:welch:wynn:1989}, spatial data modeling \cite{banerjee:gelfand:ofinley:sang:2008,cressie:2015} and geostatical kriging \cite{matheron:1973,stein:2012}, just to mention a few. 

In our paper we take a frequentist perspective and assume that the data in \eqref{eq: nonpara reg} is generated from a fixed ``true'' regression function $f_0$ and we are interested how well this functional parameter $f_0$ can be recovered from the posterior distribution as the sample size increases and how accurate and reliable is the uncertainty quantification resulting from the posterior distribution. These questions have been one of the main focus of the Bayesian asymptotics literature. Theoretical underpinning  were derived for a wide range of Gaussian processes priors and models also beyond the regression framework, both for recovery  \cite{vzanten:vdv:07,vvvz08,vzanten:vdv:11, knapik, bhattacharya:2014,rousseau:szabo:2015:main,yang:2016} and uncertainty quantification  \cite{ castillo:nickl:2013, szabo:vdv:vzanten:13,belitser:2014, sniekers:2015,rousseau:szabo:16:main,yoo:2016,bhattacharya:pati:yun:17,ray2017,hadji:szabo:2021}.

The main drawback of GP regression, however, is its computational complexity. To compute the posterior covariance, one needs to invert an $N\times N$ matrix, which in general requires a cubic algorithm in $N$. Besides, the memory requirement is also substantial, $O(N^2)$. These render the problem computationally infeasible for large data sets. Scaling up the algorithm became of particular interest in recent years and lead to the study of various approximation methods, for instance by considering sparse approximations of the matrices \cite{gibbs:poole:stockmeyer:1976,saad:1990,quinonero:rasmussen:2005}, variational Bayes approximations \cite{titsias:2009, burt2019, nieman2021contraction} or distributed methods. In distributed learning and statistical methods, the data are divided over several machines which process the data locally and then the local results are aggregated at a central server or machine. These methods, beside scaling up the computations, also help protecting privacy as the whole data set doesn't have to be stored at a single, central database. 

Various distributed methods were proposed for Gaussian process regression, recent examples include Consensus Monte Carlo \cite{scott:blocker:et.al:16}, WASP \cite{srivastava:2015a}, Bayesian Committee Machine \cite{tresp2000bayesian,deisenroth:2015}, and Distributed Kriging \cite{guhaniyogi2017divide}, to mention but a few. The theoretical properties of a wide range of distributed GP methods were investigated in \cite{szabo:vzanten:19} in context of the Gaussian white noise model. Frequentist theoretical guarantees but also limitations were derived for these procedures both for recovery and uncertainty quantification. However, the Gaussian white noise model serves only as a starting, benchmark model allowing explicit analytic computations and does not imply guarantees for the more complicated, but practically more relevant Gaussian regression model, which is the focus of our paper.

The organization of the paper is as follows. In Section \ref{sec: GP fram} we recall the distributed multivariate random design regression model where we carry out our analysis. In Section \ref{sec: distributed} we introduce various distributed Bayesian methods, derive optimal contraction rates for them and characterize the frequentist coverage of the corresponding (inflated) credible sets.  We demonstrate the applicability of our general, abstract results in a numerical analysis using synthetic data sets in Section \ref{sec:simulation}. We summarize our results and discuss open questions in Section \ref{sec:discussion}. The proofs are deferred to Section \ref{sec: proofs}, Section \ref{sec:corollaries} and the Appendix.

\subsection{Notations}

We use the notation $\mathbb{D}_N=(Y_i,X_i)_{i=1,...,N}$ for the observations and $P_0$ and $E_0$ for the probability measure and expected value corresponding to the underlying regression function $f_0$. For matrices $A \in\mathbb{R}^{d\times n}$ and $B\in\mathbb{R}^{d\times n'}$, let $K(A,B)$, denote the $n\times n'$ matrix of $(K(A_{\cdot i},B_{\cdot j}))_{1\leq i\leq n,\, 1\leq j\leq n'}$, where $K$ stands for the covariance kernel of the prior.

Furthermore, let $\|\cdot\|$ stands for the $L_2$-norm. For two sequences $a_n,b_n$ we write $a_n\lesssim b_n$ if there exists a constant $C>0$ such that $a_n/b_n\leq C$. We denote by $a_n\asymp b_n$ if $a_n\lesssim b_n$ and $b_n\lesssim a_n$ holds simultaneously.  In the manuscript $C$ and $c$ denote constants not depending on $n$ and their values might change from line to line.

\section{GP regression framework}\label{sec: GP fram}
\subsection{Standard (non-distributed) setting}\label{subsec: mod}


In our analysis we consider the model \eqref{eq: nonpara reg} and for simplicity we take $\mathcal{X}=[0,1]^d$, assume that the design points are uniformly distributed $X_i\stackrel{iid}{\sim} U[0,1]^d$ and $\sigma^2\gtrsim 1$ to be known. We endow the functional parameter $f$ with a stationary Gaussian process prior $GP(0,K)$, where $K:\mathbb{R}^d\times\mathbb{R}^d\mapsto\mathbb{R}$ is a positive definite, stationary kernel. Then by conjugacy, the posterior distribution of $f$ is also a Gaussian process, i.e. $f|\mathbb{D}_N\sim$GP$(\hat{f}_N,\hat{C}_N)$, where for any $x,x'\in[0,1]^d$,
\begin{align}
\label{eq: postmean}
&\hat{f}_N(x)=K(x,\mathbb{X})[K(\mathbb{X},\mathbb{X})+\sigma^2I_N]^{-1}\mathbb{Y},\\
\label{eq: postcov}
&\hat{C}_N(x,x')=K(x,x')-K(x,\mathbb{X})[K(\mathbb{X},\mathbb{X})+\sigma^2 I_N]^{-1}K(\mathbb{X},x'),
\end{align}
with $\mathbb{X}\in [0,1]^{d\times N}$, $\mathbb{Y}\in \mathbb{R}^{N}$  are the collection of design points and observations, respectively, and $I_N$ denotes the $N\times N$ identity matrix, see also Chapter 2 of \cite{rasmussen:williams:2006}.

We assume that the eigenfunctions $\{\psi_j\}_{j\in\mathbb{N}^d}$ of the above covariance kernel $K$ factorize, i.e.
\begin{align}\label{eq: eigenfunctions}
\psi_j=\prod_{k=1}^d \psi_{j_k}, j\in\mathbb{N}^d,
\end{align}
where $\{\psi_{j_k}\}_{j_k\in\mathbb{N}}$ are the eigenfunctions corresponding to the one dimensional kernel on $[0,1]$.  We further assume that the eigenfunctions of the kernel $K$ are bounded. 
\begin{assumption}\label{ass: bound eigenf}
There exists a global constant $C_{\psi}>0$ such that the eigenfunctions $\{\psi_j\}_{j\in\mathbb{N}^d}$ of $K$ satisfy $|\psi_j(t)|\leq C_{\psi}$ for all $j\in\mathbb{N}^d, t\in\mathcal{X}$.
\end{assumption}

The corresponding eigenvalues of $K$ are then of the form
\begin{align}\label{eq: eigenvalues}
\mu_j=\prod_{k=1}^d \mu_{j_k}, j\in\mathbb{N}^d,
\end{align}
with $\{\mu_{j_k}\}_{j_k\in\mathbb{N}}$ the eigenvalues of the $k$-th component of the kernel \cite{berlinet:ta:2004}. Although our results hold more generally, as specific examples we consider polynomially and exponentially decaying eigenvalues.
\begin{assumption}\label{ass: bound eigenvalues}
The one dimensional eigenvalues $\mu_i$, $i\in\mathbb{N}$ are either
\begin{itemize}
\item Polynomially decaying: 
\begin{equation}
C^{-1} i^{-1-2\alpha/d}\leq  \mu_i\leq C i^{-1-2\alpha/d},\label{assump:poly}
\end{equation}
 for some $\alpha,C>0$, or
\item Exponentially decaying: 
\begin{equation}
C^{-1} b e^{-a i }\leq  \mu_i\leq C b e^{-a i },\label{assump:exp}
\end{equation}
for some $a,b,C>0$.
\end{itemize}
\end{assumption}


In nonparametric statistics, it is common to assume that the underlying functional parameter of interest belongs to some regularity class. In our analysis we consider Sobolev-type of regularity classes defined with the basis $\psi_j$, i.e. for any $\beta>0$ and $B>0$, define as in \cite{benyi:oh:13,hunter:13,cobos:kuhn:sickel:15} the function space
\begin{align}\label{Sobolev func}
\Theta^{\beta}(B)=\{f=\sum_{j\in\mathbb{N}^d}f_j\psi_j\in L_2([0,1]^d):\sum_{j\in\mathbb{N}^d}\Big(\sum_{i=1}^d j_i^2\Big)^{\beta}f_j^2\leq B^2\}.
\end{align}
For the Fourier basis or the basis corresponding to the Mat\'ern covariance kernel, $\Theta^{\beta}(B)$ is equivalent to $\beta$-smooth Sobolev balls and are known as \textit{isotropic Sobolev spaces}, see \cite{cobos:kuhn:sickel:15}.

The frequentist properties of Gaussian process priors for recovery are well understood in the literature. It was shown in various specific examples and choices of priors that  for appropriately scaled Gaussian priors the corresponding posterior can recover the underlying functional parameter of interest $f_0\in\Theta^{\beta}(B)$ with the optimal minimax estimation rate $N^{-\beta/(d+2\beta)}$, see for instance \cite{vzanten:vdv:07,vvvz08,vzanten:vdv:11}.  Another, from a practical perspective very appealing property of Bayesian methods is the built-in uncertainty quantification. Bayesian credible sets accumulate prescribed (typically $95\%$) posterior mass and can take various forms. In our analysis we consider $L_2$ credible balls, i.e. we define the credible set as $\hat{B}_N=\{f:\|f-\hat{f}_N\|\leq r_{\gamma}\}$, satisfying $\Pi(f\in\hat{B}_N|\mathbb{D}_N)=1-\gamma$, for some $\gamma\in(0,1)$. Credible sets do not provide automatically valid confidence statements. In recent years the frequentist coverage properties of Bayesian credible sets were widely studied and it was shown that for appropriate choices of the prior distribution the corresponding posterior can provide reliable frequentist uncertainty quantification for functions satisfying certain regularity assumptions, see for instance \cite{szabo:vdv:vzanten:13,sniekers:2015,belitser:2014,ray2017,bhattacharya:pati:yun:17,rousseau:szabo:16:main}. However, our setting wasn't covered by these results yet.

Despite the fact that the mean (\ref{eq: postmean}) and covariance (\ref{eq: postcov}) functions can be explicitly computed, consequently solving the model, their computation requires inverting the matrix $(K(\mathbb{X},\mathbb{X})+\sigma^2I_N)$. The inversion of this $N\times N$ matrix is of $O(N^3)$ computational complexity, which rapidly explodes as $N$ grows. To speed up the computations various approximation methods were considered, our focus here lies on distributed approaches.

\section{Distributed Gaussian Process regression}\label{sec: distributed}
In distributed methods, the data are divided among multiple local machines or servers, and the computations are carried out locally, in parallel to each other. Then the outcome of the computations are transmitted to a center machine or server where they are aggregated in some way forming the final outcome of the distributed method. In the random design regression model it means that we divide the data of size $N$ over $m$ machines (we assume for simplicity that $N$ mod $m=0$), i.e. in each machine $k=1,...,m$ we observe iid pairs of random variables $(X_i^{(k)},Y_i^{(k)})\in [0,1]^d\times \mathbb{R}$, $i=1,...,n$, with $n=N/m$, satisfying
\begin{align}\label{eq: nonpara reg dist}
Y_i^{(k)}=f_0(X_i^{(k)})+\varepsilon_i^{(k)},\quad \varepsilon_i^{(k)}\stackrel{iid}{\sim}\mathcal{N}(0,\sigma^2),
\end{align}
where $f_0:[0,1]^d\mapsto\mathbb{R}$ is the unknown functional parameter of interest, and $\sigma^2>0$ the known variance of the noise. For convenience, let us introduce the notations $\mathbb{D}_{n}^{(k)}=(X_i^{(k)},Y_i^{(k)})_{i=1,...,n}$, $\mathbb{X}^{(k)}_n=(X_i^{(k)})_{i=1,...,n}$, $\mathbb{Y}^{(k)}_n=(Y_i^{(k)})_{i=1,...,n}$ for the whole data set, the design points, and observations in the $k$-th local machine, respectively. Similarly to the non-distributed setting (with only one local machine $m=1$), we assume that the true function belongs to some Sobolev-type of regularity class $f_0\in\Theta^{\beta}(B)$, for given $\beta,B>0$, see \eqref{Sobolev func}.

We consider distributed Bayesian approaches for recovering $f_0$. First, we endow the function $f_0$ in each local machine $k=1,...,m$  with a Gaussian process prior and compute the corresponding local (adjusted) posterior distribution $\Pi^{(k)}(.|\mathbb{D}_{n}^{(k)})$. Then, we transmit the $m$ local posteriors into a central machine where we aggregate them somehow into a global (adjusted) posterior $\Pi^{\dagger}_{n,m}(\cdot|\mathbb{D}_N)$. We further denote by $\hat{f}_{n}^{(k)}$ the local (adjusted) posterior mean, and by $\hat{f}_{n,m}$ the global (adjusted) posterior mean. For quantifying the uncertainty of the distributed Bayesian procedure we consider $L_2$-credible balls resulting in from the aggregated posterior distribution, i.e. let
\begin{align}
\hat{B}_{n,m,\gamma}=\{f:\|f-\hat{f}_{n,m}\|\leq r_{n,m,\gamma}\},\quad  \text{satisfying}\nonumber\\
\Pi^{\dagger}_{n,m}(f\in\hat{B}_{n,m,\gamma}|\mathbb{D}_n)=1-\gamma,\label{def:cred:aggr}
\end{align}
 for some prescribed $\gamma\in(0,1)$.

Distributed methods vary according to the way the local (adjusted) posterior distributions are computed and aggregated to obtain the global posterior. The behavior of the aggregated posterior crucially depends on the applied techniques. To demonstrate this let us consider a naive method where in each local machine we endow $f_0\in\Theta^{\beta}(B)$ with a Gaussian process prior and compute the corresponding unadjusted local posterior distribution $\Pi^{(k)}_{n}(\cdot|\mathbb{D}_{n}^{(k)})$. We consider a centered GP with polynomially decaying eigenvalues as in Assumption \ref{ass: bound eigenvalues} with regularity hyper-parameter matching the regularity of the truth $\alpha=\beta$. Note that this choice of the hyper-parameter is optimal in the non-distributed case (with only one local machine $m=1$). Then the local posteriors are aggregated to a global posterior $\Pi_{n,m}^{\dagger}(\cdot|\mathbb{D}_N)$ in the following way: a draw from the aggregated posterior is taken to be the average of a single draw from each local posteriors. We refer to this simple approach as the naive method. The theorem below shows that this method, in accordance with its name, results in sub-optimal concentration for the posterior mean and sub-optimal contraction rate for the whole posterior distribution as well.

\begin{theorem}\label{thm:counter}
Take $\beta\geq 2$ and consider the function $f_0\in\Theta_\beta(L)$ of the form $f_{0}(x)=c_L\sum_{j=1}^\infty j^{-1-2\beta}(\log j)^{-2} \psi_j(x)$, $x\in[0,1]$, for sufficiently small $c_L>0$. Then for the covariance kernel $K$ with polynomially decaying eigenvalues \eqref{assump:poly} with $\alpha=\beta$ and $d=1$, and $(\log n)^2\ll m\lesssim n^{1/(1+2\beta)}$ the corresponding naive aggregated posterior mean $\hat{f}_{n,m}$ has sub-optimal concentration and the posterior itself achieves suboptimal contraction rate, i.e. 
\begin{align}
E_0\|\hat{f}_{n,m}-f_{0}\|_2^2\geq c (\log N)^{-2}(N/m)^{-\beta/(1+2\beta)},\label{eq: bad_estimate}\\
{E_0\Pi_{n,m}^{\dagger}\Big(f:\|f-f_{0}\|_2^2\leq  c(\log N)^{-2}(N/m)^{-\beta/(1+2\beta)}|\mathbb{D}_N\Big)\to 0},\label{eq: bad_contraction}
\end{align}
for sufficiently small $c>0$, where $\hat{f}_{n,m}$ is the mean of the global posterior $\Pi_{n,m}^{\dagger}$ obtained with the naive method.
\end{theorem}
The proof is given in Section \ref{sec:thm:counter}

\subsection{Optimal Distributed Methods}\label{sec:methods}
In this paper we consider two methods, for which optimal frequentist performance were derived in context of the Gaussian white noise setting, see \cite{szabo:vzanten:19}. We investigate these methods here in the practically more relevant and technically substantially more complex nonparametric regression model. We note that in \cite{guhaniyogi2017divide} in context of the regression model an approach closely related to Method II was derived and its contraction properties were investigated for a rescaled covariance kernel with polynomially decaying eigenvalues. In our work we consider more general kernel structures and in contrast to  \cite{guhaniyogi2017divide} do not require that the functional parameter belongs to the Reproducing Kernel Hilbert Space (RKHS) of the Gaussian Process prior. Furthermore, we also derive guarantees and limitations to uncertainty quantification. Therefore, our results are of different nature requiring a different approach.

\subsubsection{Method I}

\paragraph{Rescaling the priors.} In the first method, introduced by \cite{scott:blocker:et.al:16} in a parametric setting, we consider raising the prior density to the power $1/m$, which is formally equivalent to multiplying the kernel $K$ by $m$, i.e. the adjusted kernel takes the form  $K^I:=mK$. Then the eigenvalues of the kernel $K^I$ are $\{\mu_j^I\}_{j\in\mathbb{N}^d}=\{m\mu_j\}_{j\in\mathbb{N}^d}$. 
Hence, in view of (\ref{eq: nonpara reg dist}) the posterior distribution, for each machine $k=1,...,m$, is also a Gaussian process $f|\mathbb{D}_{n}^{(k)}\sim$GP$(\hat{f}_{n}^{(k)},\hat{C}_{n}^{(k)})$ with
\begin{align*}
&\hat{f}_{n}^{(k)}(x)=K(x,\mathbb{X}^{(k)}) [K(\mathbb{X}^{(k)},\mathbb{X}^{(k)})+m^{-1}\sigma^2I_{n}]^{-1}\mathbb{Y}^{(k)},\\
&\hat{C}_{n}^{(k)}(x,x')=m\Big(K(x,x')-{K}(x,\mathbb{X}^{(k)})[{K}(\mathbb{X}^{(k)},\mathbb{X}^{(k)})+m^{-1}\sigma^2 I_{n}]^{-1}{K}(\mathbb{X}^{(k)},x')\Big).
\end{align*}

\paragraph{Averaging the local draws.}
A draw from the global posterior is generated by first drawing a single sample from each local posteriors and then taking the averages of these draws over all machines. Since the data sets and the priors in the local machines are independent, the so generated average of the local posteriors is also a Gaussian process with mean $\hat{f}_{n,m}^{I}=m^{-1}\sum_{k=1}^m\hat{f}_{n}^{(k)}$ and covariance kernel $\hat{C}_{n,m}^{I}=m^{-2}\sum_{k=1}^m\hat{C}_{n}^{(k)}$, where $\hat{f}_{n}^{(k)}$ and $\hat{C}_{n}^{(k)}$ denote the posterior mean and covariance functions in the $k$th local machine, respectively. 

\subsubsection{Method II}

\paragraph{Rescaling the likelihood.}
In the second method proposed by \cite{srivastava:2015a}, we adjust the local likelihood by raising its power to $m$ in every machine, which is equivalent to rescaling the variance of the observations by a factor $m^{-1}$. Then, by elementary computations similar to \eqref{eq: postmean} and \eqref{eq: postcov}, we obtain that for each machine, the posterior distribution is $GP(\hat{f}_{n}^{(k)},\hat{C}_{n}^{(k)})$, with
\begin{align*}
&\hat{f}_{n}^{(k)}(x)=K(x,\mathbb{X}^{(k)})[K(\mathbb{X}^{(k)},\mathbb{X}^{(k)})+m^{-1}\sigma^2I_{n}]^{-1}\mathbb{Y}^{(k)},\\
&\hat{C}_{n}^{(k)}(x,x')=K(x,x')-K(x,\mathbb{X}^{(k)})[K(\mathbb{X}^{(k)},\mathbb{X}^{(k)})+m^{-1}\sigma^2 I_{n}]^{-1}K(\mathbb{X}^{(k)},x').
\end{align*}

\paragraph{Wasserstein barycenter.}
This approach consists in aggregating the local posteriors by computing their Wasserstein barycenter. The 2-Wasserstein distance $W_2^2(\mu,\nu)$ between two probability measures $\mu$ and $\nu$ is defined as
$$W_2^2(\mu,\nu):=\inf_{\gamma}\int\int\|x-y\|_2^2\gamma(dx,dy),$$
where the infimum is taken over all measures $\gamma$ with marginals $\mu$ and $\nu$. The corresponding 2-Wasserstein barycenter of $m$ probability measures $\mu_1,...,\mu_m$ is defined by
$$\bar{\mu}=\arg\min_{\mu}\frac{1}{m}\sum_{k=1}^mW_2^2(\mu,\mu_k),$$
where the minimum is taken over all probability measures with finite second moments. In view of Theorem 4 in \cite{mallasto:feragen:17}, the global posterior is a Gaussian process with mean $\hat{f}_{n,m}^{II}$ and covariance $\hat{C}_{n,m}^{II}$ satisfying
\begin{align*}
\hat{f}_{n,m}^{II}&=\frac{1}{m}\sum_{k=1}^m\hat{f}_{n}^{(k)},\\
\hat{C}_{n,m}^{II}&=\frac{1}{m}\sum_{k=1}^m\big((\hat{C}_{n,m}^{II})^{1/2}\hat{C}_{n}^{(k)}(\hat{C}_{n,m}^{II})^{1/2}\big)^{1/2}.
\end{align*}
In particular, the posterior variance function is
$$\Var_{n,m}^{II}[f(x)|\mathbb{D}_N]=\frac{1}{m}\sum_{k=1}^m\Var[f(x)|\mathbb{D}^{(k)}_{n}]$$
for all $x\in\mathcal{X}$.

\subsection{Posterior contraction rate}
We show that the above proposed distributed methods (i.e. Methods I- II) provide optimal recovery of the underlying functional parameter of interest. The methods result in different global posteriors which can have different finite sample size behavior, but their asymptotic properties are similar.

\begin{theorem}\label{thm: contr:general}
Let $\beta,B>0$, $K$ a kernel with eigenvalues $(\mu_j)_{j\in\mathbb{N}^d}$ satisfying $|\{j\in \mathbb{N}^{d}:\,  \mu_j N \geq \sigma^2\}|\leq N$ and corresponding eigenfunctions satisfying Assumption \ref{ass: bound eigenf}. Furthermore, let
\begin{align}\label{eq: eigenvalues equII}
\nu_j=\frac{\mu_jN}{\sigma^2+\mu_jN},\quad \text{for all $j\in\mathbb{N}^d$},
\end{align}
and $\tilde{P}$ a linear operator defined as $\tilde{P}(f):=\sum_{j\in\mathbb{N}^d}(1-\nu_j)f_j\psi_j$ for all $f\in L^2(\mathcal{X})$. Then
\begin{align}
\label{eq: bound post mean}
E_0\|\hat{f}_{n,m}-f_0\|_2^2\lesssim \|\tilde{P}(f_0)\|_2^2+\frac{\sigma^2}{N}\sum_{j\in\mathbb{N}^d}\nu_j^2+ \delta_{N},\\
\label{eq: contr rate}
E_0\Pi^{\dagger}_{n,m}\Bigg(\|f-f_0\|_2^2>M_N\Big(\|\tilde{P}(f_0)\|_2^2+\frac{\sigma^2}{N}\sum_{j\in\mathbb{N}^d}\nu_j+\delta_{N}\Big)|\mathbb{D}_N\Bigg)\to 0,
\end{align}
for arbitrary sequence $M_N$ tending to infinity, where $\hat{f}_{n,m}$ is the mean of the global posterior $\Pi^{\dagger}_{n,m}(.|\mathbb{D}_N)$ obtained with either Methods $I-II$ and 
\begin{align}
\delta_{N}=\inf\Big\{ N\sum_{j\in\mathbb{N}^d}\nu_j^2 \sum_{\ell\in \mathcal{I}^c}\mu_\ell:\, \mathcal{I}\subset \mathbb{N}^d, |\mathcal{I}|\leq n\big(m\sum_{j\in\mathbb{N}^d}\nu_j^2\big)^{-1}\Big\}\label{def:deltaN}
\end{align}
is a (typically) negligible technical term.
\end{theorem}

The proof of the theorem is deferred to Section \ref{sec: contr:general}.

First we note that the condition $|\{j\in \mathbb{N}^{d}:\,  \mu_j N \geq \sigma^2\}|\leq N$ is very mild and is satisfied by the eigenvalues considered in Assumption \ref{ass: bound eigenvalues}. The sequence $(\nu_j)_{j\in\mathbb{N}}$ can be thought of as the population eigenvalues of the posterior. Next note that the bound \eqref{eq: bound post mean} has two main components. The first term $\|\tilde{P}(f_0)\|_2^2$ measures how close $f_0$ is (in $L_2$-norm) to its convolution with the eigenvalues $(\nu_j)_{j\in\mathbb{N}^d}$, hence it accounts for the bias of the estimator. In the meanwhile the second term $(\sigma^2/N)\sum_{j\in\mathbb{N}^d}\nu_j^2$ can be thought of as the variance term. In a similar fashion, the contraction rate \eqref{eq: contr rate} has also two main components: $\|\tilde{P}(f_0)\|_2^2$ and $(\sigma^2/N)\sum_{j\in\mathbb{N}^d}\nu_j$, where the former is the squared bias while the latter is the expected value of the posterior variance under the true parameter. The remaining $\delta_{N}$ term is of technical nature. It bounds the tail behaviour of the eigen-decomposition of the variance of the posterior mean. This term is shown to be negligible in our examples. Since all the above terms are related to the kernel $K$, explicit bounds on the expectation of $\|\hat{\theta}_n-\theta_0\|_2$, as well as explicit posterior contraction rates of the global posterior $\Pi^{\dagger}_{n,m}(.|\mathbb{D}_N)$, can be achieved for specific choices of the kernels.


\begin{corollary}\label{cor: contraction:poly}\textbf{(Polynomial)}
For given $B>0$ and $\beta\geq 3d/2$, assume that the covariance kernel $K$ satisfies Assumptions \ref{ass: bound eigenf} and \eqref{assump:poly} with $\alpha=\beta$. Then for $m=o(N^{\frac{2\beta-3d}{4\beta}})$ 
the aggregated posterior distribution $\Pi^{\dagger}_{n,m}(.|\mathbb{D}_N)$ and the corresponding aggregated posterior mean $\hat{f}_{n,m}$ resulting from either of the Methods $I-II$ achieve the minimax convergence rate up to a logarithmic factor, i.e.
$$\sup_{f_0\in\Theta^{\beta}(B)}E_0\|\hat{f}_{n,m}-f_0\|_2^2\lesssim(N/\sigma^2)^{-2\beta/(2\beta+d)}\log^{d-1}(N/\sigma^2)$$
and for all sequences $M_N\to +\infty$,
$$\sup_{f_0\in\Theta^{\beta}(B)}E_0\Pi^{\dagger}_{n,m}(f:\|f-f_0\|_2>M_N(N/\sigma^2)^{-\beta/(2\beta+d)}(\log(N/\sigma^2))^{(d-1)/2}|\mathbb{D}_N)\to 0.$$
\end{corollary}

The proof is given in Section \ref{sec: proof_contraction:poly}.

\begin{corollary}\label{cor: contraction:exp}\textbf{(Exponential)}
For given $B>0$ and $\beta\geq d/2$ assume that the covariance kernel $K$ satisfies Assumptions \ref{ass: bound eigenf} and \eqref{assump:exp} with rescaling parameter $a=(\sigma^2/N)^{1/(2\beta+d)}\log(N/\sigma^2)$ and $b=1$. Then for $m=o(N^{\frac{2\beta-d}{2(2\beta+d)}})$ the aggregated posterior distribution $\Pi^{\dagger}_{n,m}(.|\mathbb{D}_N)$ and the corresponding aggregated posterior mean $\hat{f}_{n,m}$ resulting from either of the Methods $I-II$ achieve the minimax convergence rate, i.e.
$$\sup_{f_0\in\Theta^{\beta}(B)}E_0\|\hat{f}_{n,m}-f_0\|_2^2\lesssim(N/\sigma^2)^{-2\beta/(2\beta+d)},$$
and for all sequences $M_N\to +\infty$,
$$\sup_{f_0\in\Theta^{\beta}(B)}E_0\Pi^{\dagger}_{n,m}\left(f:\|f-f_0\|_2>M_N(N/\sigma^2)^{-\beta/(2\beta+d)}|\mathbb{D}_N\right)\to 0.$$
\end{corollary}

The proof is given in Section \ref{sec: proof_contraction:exp}. We note that the conditions on $\beta$ and $m$ in both corollaries follow from the remaining technical term $\delta_{N}$. These conditions are not optimized and are of technical nature.

\section{Distributed uncertainty quantification}

In the following, we study the frequentist coverage properties of the $L_2$ credible balls defined in \eqref{def:cred:aggr} resulting from Method I. For convenience we allow some additional flexibility by allowing the credible balls to be blown up by a constant factor $L>0$, i.e. we consider balls
\begin{align*}
\hat{B}_{n,m,\gamma}(L)=\{f\in L_2(\mathcal{X}):\, \|f-\hat{f}_{n,m}\|_2\leq L r_{n,m,\gamma}\},
\end{align*}
where for the choice $L=1$ we get back our original credible ball \eqref{def:cred:aggr}. The frequentist validity of $\hat{B}_{n,m,\gamma}(L)$ will be established in two steps: first we approximate the centered posterior measure $f-\hat{f}_{n,m}|\mathbb{D}_N$ and second we study the asymptotic behavior of the radius, and the bias and variance of the posterior mean, corresponding to the approximated posterior.

In the non-distributed case (i.e. $m=1$), the posterior distribution can be approximated by an auxiliary Gaussian process. For the GP posterior $f-\hat{f}_N|\mathbb{D}_N\sim$GP$(0,\hat{C}_N)$, the covariance kernel $\hat{C}_N$ given in \eqref{eq: postcov} is hard to analyze due to its dependence  on $\mathbb{X}$. Against this background, following the idea in \cite{bhattacharya:pati:yun:17}, we define a population level GP $\hat{W}\sim$GP$(0,\tilde{C}_n)$, where $\tilde{C}_N(x,x')=\sigma^2/N\sum_{j\in\mathbb{N}^d}\nu_j\boldsymbol{\psi}_j(x)\boldsymbol{\psi}_j(x')$, and show that the two kernels are close with respect to the $L_2$-norm. Then using this result we can provide the following frequentist coverage results for the credible balls.

\begin{theorem}\label{th: unc quant}
Let $\beta,B>0$, $K$ a kernel with eigenvalues $(\mu_j)_{j\in\mathbb{N}^d}$ satisfying $|\{j\in \mathbb{N}^{d}:\,  \mu_j N \geq \sigma^2\}|\leq N$ and corresponding eigenfunctions satisfying Assumption \ref{ass: bound eigenf}. Furthermore, assume that $N\delta_{N}/\sum_{j\in\mathbb{N}^d}\nu_j=o(1)$, where the (typically) negligible technical term $\delta_N$ was defined in \eqref{def:deltaN}. Then in case the bias term $\|\tilde{P}(f_0)\|_2$ satisfies that
\begin{align}\label{cond:biasUB}
\frac{N}{\sigma^2}\frac{\|\tilde{P}(f_0)\|_2^2}{\sum_{j\in\mathbb{N}^d}\nu_j}\leq c
\end{align}
for some $c\geq0$, the frequentist coverage of the (inflated) credible set resulting from Method I tends to one, i.e. for arbitrary $L_n\rightarrow \infty$
$$P_0(f_0\in\hat{B}_{n,m,\gamma}(L_n))\stackrel{n\to \infty}{\to}  1.$$
On the other hand, if the bias term $\|\tilde{P}(f_0)\|_2$ satisfies that
\begin{align}\label{cond:biasLB}
\frac{N}{\sigma^2}\frac{\|\tilde{P}(f_0)\|_2^2}{\sum_{j\in\mathbb{N}^d}\nu_j}\stackrel{n\to \infty}{\to} \infty,
\end{align}
 then the aggregated and inflated credible set resulting from Method I has frequentist coverage tending to zero, i.e. for any $L>0$,
$$P_0(f_0\in\hat{B}_{n,m,\gamma}(L))\stackrel{n\to \infty}{\to} 0.$$
\end{theorem}

We briefly discuss the assumptions. Condition \eqref{cond:biasUB} requires that the squared bias term is dominated by the posterior variance, which is a natural and standard assumption for coverage. On the other hand condition \eqref{cond:biasLB} resulting in the lack of coverage assumes that the squared bias dominates the variance which is again natural and standard.  The assumption $N\delta_{N}/\sum_{j\in\mathbb{N}^d}\nu_j=o(1)$ is of technical nature, required to deal with the tail of the eigen-decomposition of the posterior. This condition is not optimized but it is already sufficiently general to cover our examples. The blow up constant of the credible sets are again of technical nature, it can be equivalently replaced by slightly under-smoothing the prior, see \cite{knapik}.

Below we consider specific choices of the covariance kernel $K$,  both with polynomially and exponentially decaying eigenvalues. We show below that by not over-smoothing the priors, Method I results in frequentist coverage tending to one in both examples.

\begin{corollary}\label{cor: coverage:poly}\textbf{(Polynomial)}
For given $B>0$ and $\beta\geq 3d/2$, assume that the covariance kernel $K$ satisfies Assumptions \ref{ass: bound eigenf} and \eqref{assump:poly} with $\alpha\leq\beta$. Then for $m=o(N^{\frac{2\beta-3d}{4\beta}})$
and $L_N$ tending to infinity arbitrarily slowly the aggregated posterior credible set $\hat{B}_{n,m,\gamma}(L_N)$ attains asymptotic frequentist coverage one, i.e.
$$\inf_{f_0\in\Theta^{\beta}(B)}P_0(f_0\in\hat{B}_{n,m,\gamma}(L_N))\to 1.$$
\end{corollary}

The proof is given in Section \ref{sec:cor: coverage:poly}.

\begin{corollary}\label{cor: coverage:exp}\textbf{(Exponential)}
For given $B>0$ and $\beta\geq d/2$, let us take $m=o(N^{\frac{2\beta-d}{2(2\beta+d)}})$ and assume that the covariance kernel $K$ satisfies Assumptions \ref{ass: bound eigenf} and \eqref{assump:exp} with 
$$(m/N)^{\frac{1}{2d}}(\log N)^{1-\frac{1}{2d}} \lesssim a\lesssim\big(\frac{\sigma}{N}\big)^{1/(2\beta+d)}\log N$$
 and $b=1$.
Then for $L_N$ tending to infinity arbitrarily slowly the aggregated posterior credible set $\hat{B}_{n,m,\gamma}(L_n)$ obtains asymptotic frequentist coverage one, i.e.
$$\inf_{f_0\in\Theta^{\beta}(B)}P_0(f_0\in\hat{B}_{n,m,\gamma}(L_N))\to 1.$$
\end{corollary}

The proof is given in Section \ref{sec:cor: coverage:exp}. We note that in both examples the conditions on the regularity $\beta$ and the number of machines are of technical nature and they were not optimized.

\section{Simulation study}\label{sec:simulation}
We illustrate our findings by performing a numerical analysis on synthetic data. We consider the regression model  \eqref{eq: nonpara reg} with uniform design on the unit interval $X_i\stackrel{iid}{\sim}U[0,1]$, fix $\sigma=1$ and take the underlying true functional parameter of interest as
\begin{align}\label{eq:sim:f0}
    f_0(x) = \sum_{i=3}^{\infty}\sin(i)
 i^{-\frac{3}{2}} \sqrt{2}\cos\left(\pi\left(i-\frac{1}{2}\right)x\right).
\end{align}
Note that this function belongs to any Sobolev class of regularity $\beta<1$. For computational reasons we truncate the above series at $i = 200$. We generate datasets in the distributed framework \eqref{eq: nonpara reg dist} and compare the statistical properties of the methods considered in Section \ref{sec:methods}. As prior distributions we consider both the Mat\'{e}rn and the squared exponential covariance kernels. We note that the Mat\'ern kernel has polynomially decaying eigenvalues, while the squared exponential (under Gaussian design on the real line) has exponentially decaying ones, see for instance \cite{rasmussen:williams:2006}. Since the true function is $\beta=1$ smooth, we choose a matching regularity hyper-parameter $\nu = \beta - 1/2= 1/2$ in the Mat\'{e}rn covariance kernel. Furthermore, we consider the rescaled version of the squared exponential covariance kernel  $k(x_i, x_j) = e^{-(x_i - x_j)^2  \tau_N/2}$ with rescaling parameter $\tau_N=10 N^{1/(1+2\beta)}=10N^{1/3}$. All our code is written in Python and run on an Intel Core i5-10300H CPU.

First we consider the Mat\'ern covariance kernel and choose a matching regularity of one to the underlying true function $f_0$ in \eqref{eq:sim:f0}. We start by demonstrating in Figure \ref{fig:naive} that the naive averaging method described above Theorem \ref{thm:counter} has indeed suboptimal statistical performance. We consider increasing sample sizes $N = 500, 1000, 10000$, $50000$, while keeping $m = 100$ fixed. The posterior mean is plotted in blue, the true function in black and the 95\% closest out of 1000 posterior draws in $L_2$-distance  to the posterior mean in gray.  In accordance with the theoretical results, this approach performs poorly both for recovering the underlying true function and for quantifying the remaining uncertainty of the procedure.

\begin{figure}
\begin{center}
  \begin{subfigure}[b]{0.4\linewidth}
    \centering
    \includegraphics[width=0.9\linewidth]{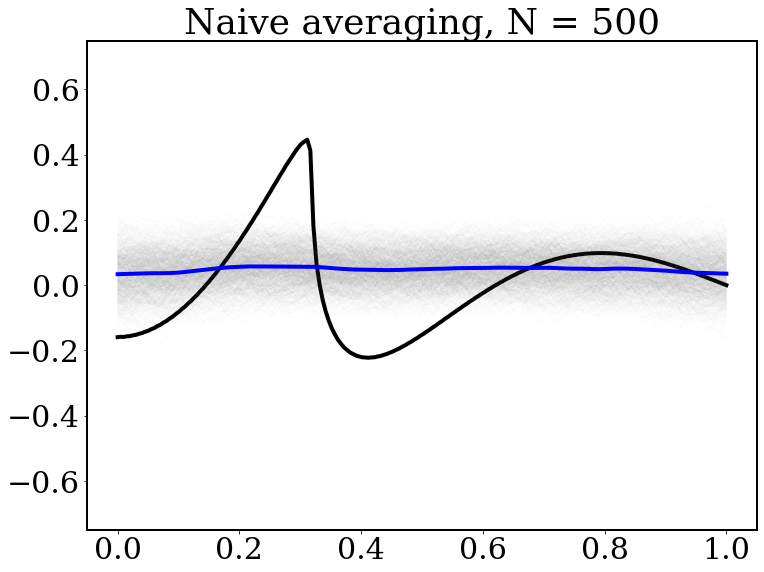} 
  \end{subfigure}
  \begin{subfigure}[b]{0.4\linewidth}
    \centering
    \includegraphics[width=0.9\linewidth]{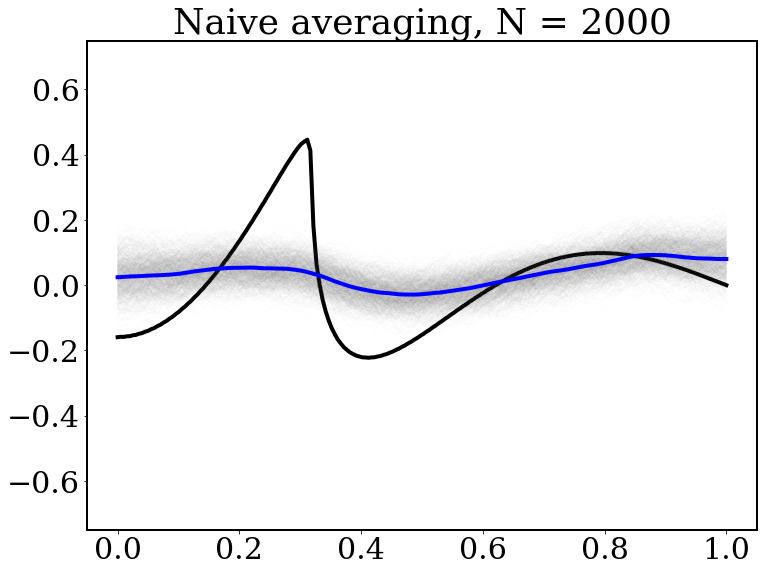} 
  \end{subfigure} 
  \begin{subfigure}[b]{0.4\linewidth}
    \centering
    \includegraphics[width=0.9\linewidth]{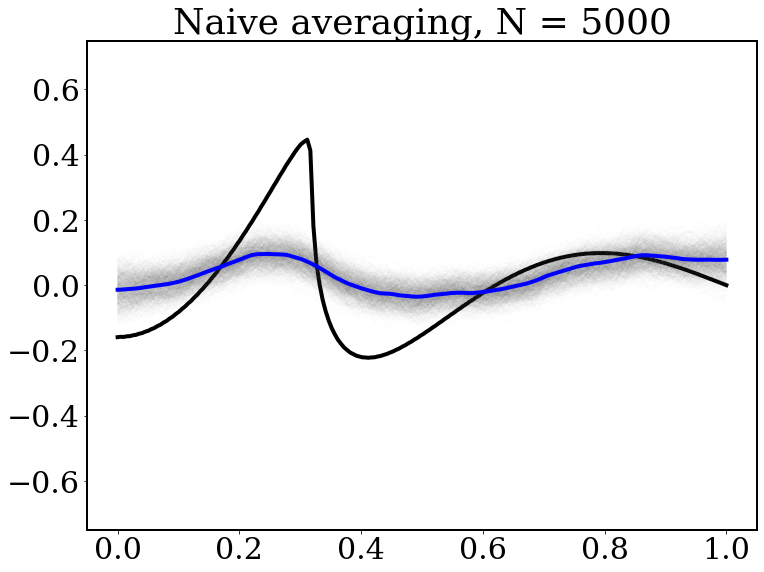} 
  \end{subfigure}
  \begin{subfigure}[b]{0.4\linewidth}
    \centering
    \includegraphics[width=0.9\linewidth]{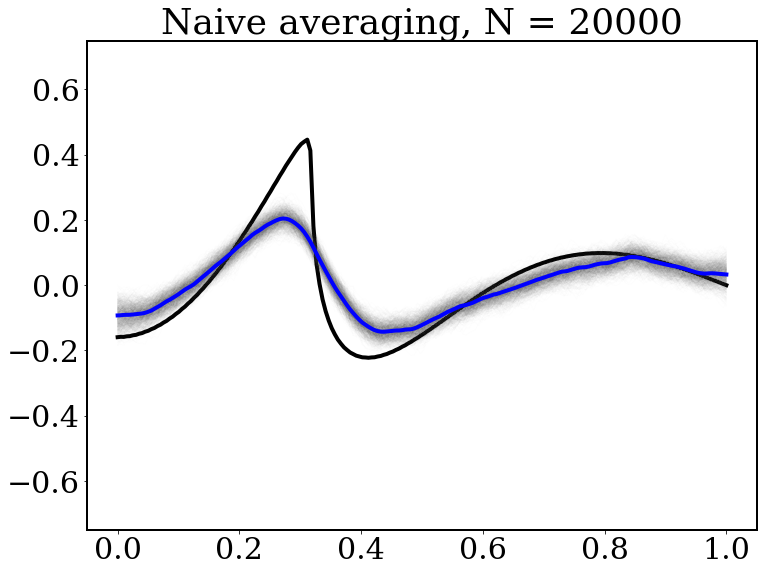} 
  \end{subfigure} 
  \caption{The naive averaging method. Increasing sample sizes from $N=500$ to $N=20000$ are considered. The true function $f_0$ is plotted by black, the posterior mean by blue and we plot the 95\% closest draws to the posterior mean from the posterior in $L_2$-distance out of 1000, indicating the $L_2$-credible sets.}
  \label{fig:naive} 
\end{center}
\end{figure}

Next we demonstrate that Methods I and II have optimal statistical performance. First we consider Method I in Figure~\ref{fig:nincreasematern} and take sample sizes $N = 500, 1000, 10000$, $50000$, while keeping $m = 100$ fixed. The posterior mean is plotted again in blue, the true function in black and the 95\% closest out of 1000 posterior draws in $L_2$-distance to the posterior mean in gray. The figure nicely illustrates that as the sample size increases the estimation accuracy will be better and that in all cases we achieve reliable frequentist coverage. 

\begin{figure} 
\begin{center}
  \begin{subfigure}[b]{0.4\linewidth}
    \centering
    \includegraphics[width=0.9\linewidth]{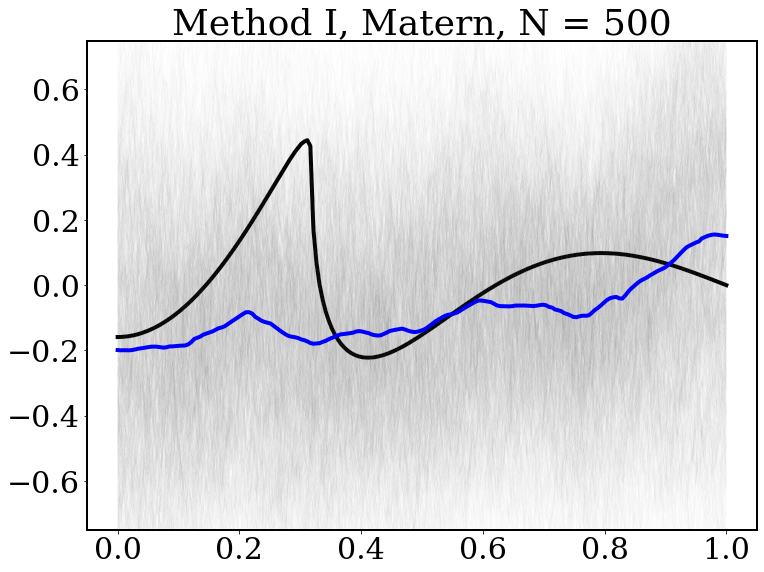} 
    \vspace{4ex}
  \end{subfigure}
  \begin{subfigure}[b]{0.4\linewidth}
    \centering
    \includegraphics[width=0.9\linewidth]{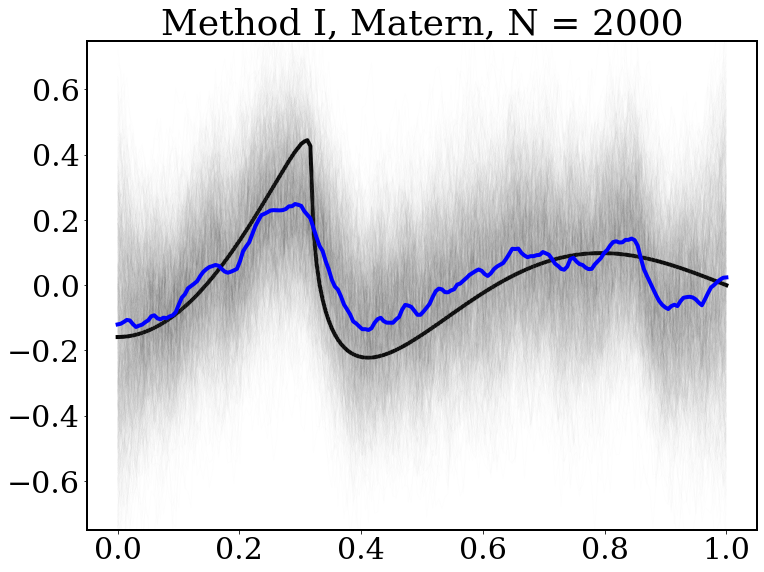} 
    \vspace{4ex}
  \end{subfigure} 
  \begin{subfigure}[b]{0.4\linewidth}
    \centering
    \includegraphics[width=0.9\linewidth]{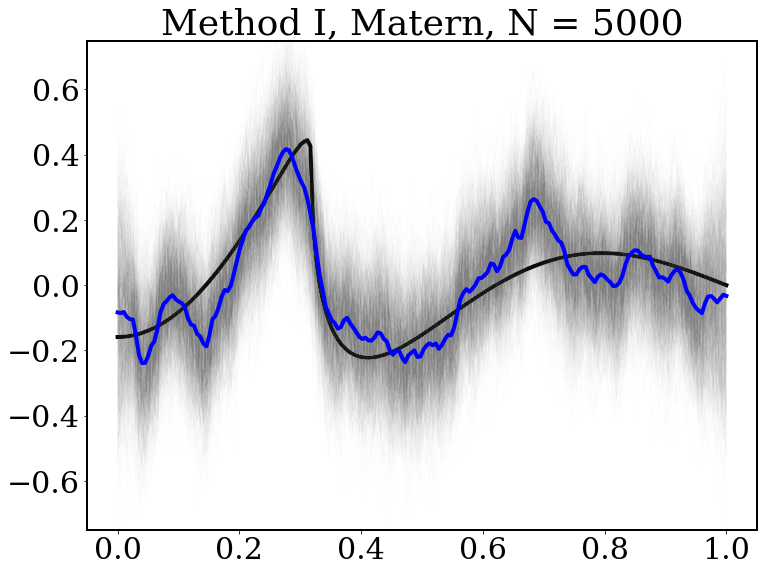} 
  \end{subfigure}
  \begin{subfigure}[b]{0.4\linewidth}
    \centering
    \includegraphics[width=0.9\linewidth]{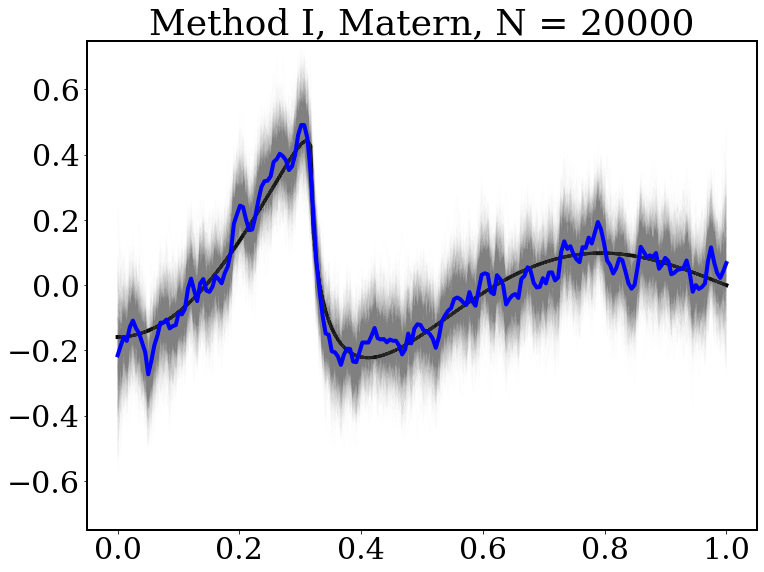} 
  \end{subfigure} 
  \caption{Method I with Mat\'ern covariance kernel. Increasing sample sizes from $N=500$ to $N=20000$ are considered. The true function $f_0$ is plotted in black, the posterior mean in blue and we plot the 95\% closest draws from the posterior to the posterior mean in $L_2$-distance out of 1000 samples in gray, indicating the $L_2$-credible sets.}
  \label{fig:nincreasematern} 
\end{center}
\end{figure}

Then we compare the finite sample size  behavior of the Methods I and II for both the Mat\'ern and squared exponential covariance kernels in Figure \ref{fig:maternmethodcomparison}. We set the total sample size $N=5000$ and take $m=100$ machines. Although the resulting approximations are different, in all cases we get good approximation and reliable uncertainty quantification.

\begin{figure}
\begin{center}
  \begin{subfigure}[b]{0.4\linewidth}
    \centering
    \includegraphics[width=0.9\linewidth]{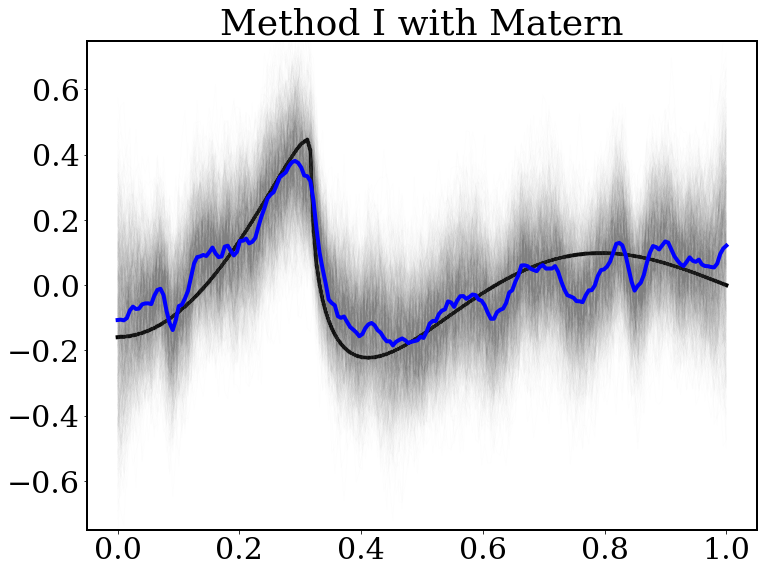} 
    \vspace{4ex}
  \end{subfigure}
  \begin{subfigure}[b]{0.4\linewidth}
    \centering
    \includegraphics[width=0.9\linewidth]{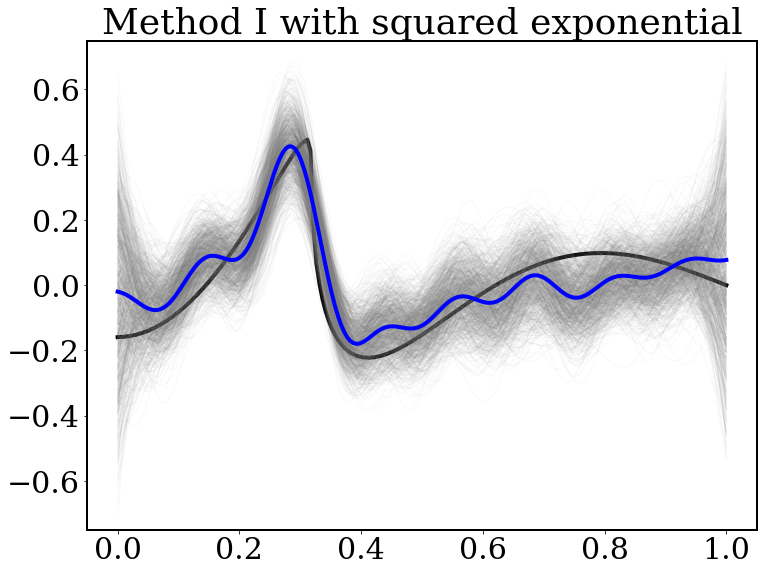} 
    \vspace{4ex}
  \end{subfigure} 
  \begin{subfigure}[b]{0.4\linewidth}
    \centering
    \includegraphics[width=0.9\linewidth]{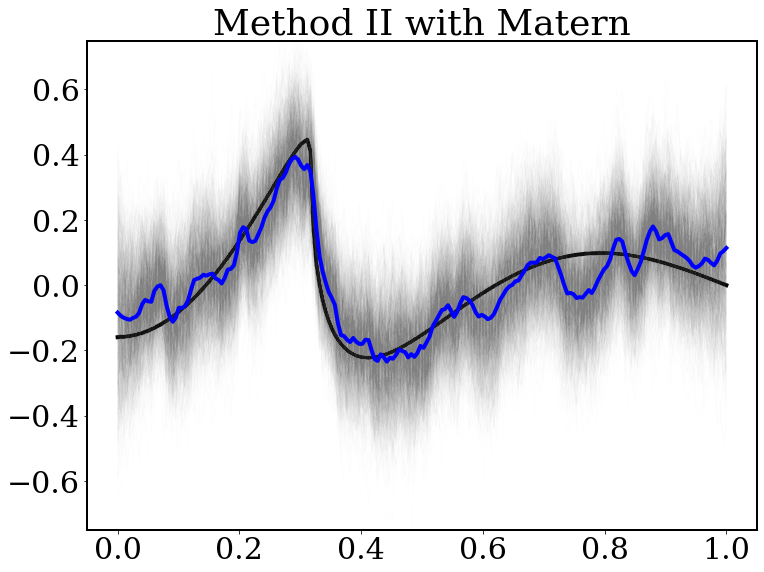} 
  \end{subfigure}
  \begin{subfigure}[b]{0.4\linewidth}
    \centering
    \includegraphics[width=0.9\linewidth]{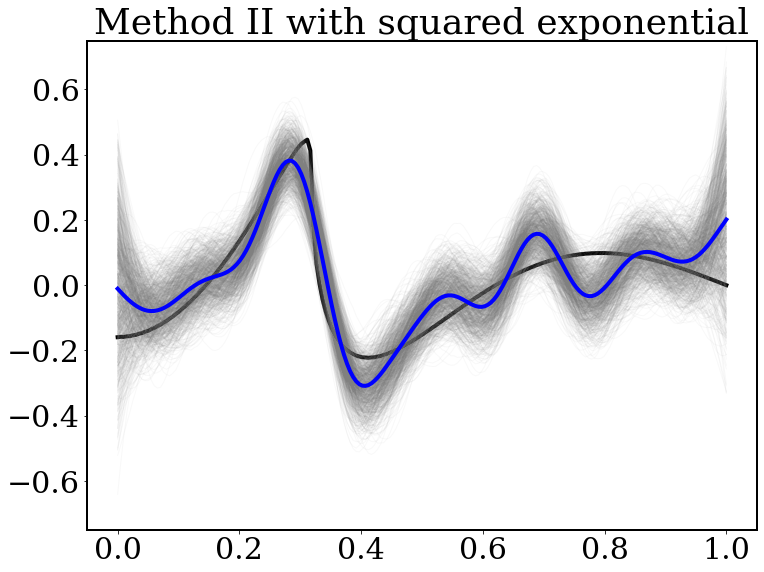} 
  \end{subfigure} 
  \caption{Comparison of Methods I and II for the covariance Mat\'ern and squared exponential kernels. We take $N = 5000$ and $m = 100$, the true function is set to be \eqref{eq:sim:f0} and it is plotted in black. The posterior mean is drawn in blue and we plot the 95\% closest draws the posterior to the posterior mean from  in $L_2$-distance out of 1000 samples in gray, indicating the $L_2$-credible sets.}
  \label{fig:maternmethodcomparison} 
\end{center}
\end{figure}

Finally, in Table \ref{table:numres} we report the average mean squared errors  of the posterior means and the corresponding standard deviations over 100 experiments. We considered both Methods I and II and both the Mat\'ern and squared exponential covariance kernels. We have set the number of machines $m=100$, while we considered increasing sample sizes $N = 500, 2000, 5000, 20000$. One can see that the MSE of the methods are rather similar. We have also computed the coverage of the credible sets. We took 1000 draws from the posterior and compute the $95\%$ percentile of the $L_2$-distance of these draws to the posterior mean. If the $L_2$-distance of $f_0$ to the posterior mean is smaller than the $95\%$ percentile, we consider the truth to be inside of the $L_2$ credible set. In Table \ref{table:coverage} we report the proportion the true function is inside of the credible sets. One can observe that in all cases we get good coverage, although the coverage property decreases as the sample size increases in the case of the squared exponential covariance kernel especially for Method II. A possible explanation for this is that the rescaling factor of $\tau_N=10N^{1/3}$ does not reach an appropriate bias-variance trade-off and a (possibly by a $\log N$ factor) larger rescaling should be applied.

\begin{table}[!ht]
    \begin{center}
    \begin{tabular}{|l|l|l|l|l|l|}
    \hline
        Method$\backslash N$ & 500 & 2000 & 5000 & 20000 \\ \hline
        Method I (Mat\'{e}rn) & 0.020(0.074) & 0.008 (0.050) & 0.005 (0.034) & 0.002 (0.021) \\ \hline
        Method II (Mat\'{e}rn) & 0.021(0.080) & 0.008 (0.056)  & 0.005 (0.035) & 0.002 (0.020) \\ \hline
Method I (Squared exp.) & 0.024 (0.094) & 0.014 (0.070) & 0.006 (0.043) & 0.002 (0.019) \\ \hline
        Method II (Squared exp.) & 0.025 (0.110) & 0.013 (0.063)  & 0.006 (0.044) & 0.002 (0.020) \\ \hline
    \end{tabular}
    \end{center}
  \captionof{table}{Average mean squared errors of the posterior means and corresponding standard deviations out of 100 runs. Increasing sample size is considered from $N=500$ to $N=20000$, while the number of machines are kept fixed at $m=100$. Both methods and both the Mat\'ern kernel and the rescaled squared exponential covariance kernels are considered. \label{table:numres}}
\end{table}

\begin{table}[!ht]
    \begin{center}
    \begin{tabular}{|l|l|l|l|l|l|}
    \hline
        Method$\backslash N$ & 500 & 2000 & 5000 & 20000 \\ \hline
        Method I (Mat\'{e}rn) & 1 &1 & 1 & 1 \\ \hline
        Method II (Mat\'{e}rn) & 1 & 1  & 1 & 1 \\ \hline
Method I (Squared exp.) & 1 & 1 & 0.99 & 0.85 \\ \hline
        Method II (Squared exp.) & 1 & 1  & 0.9 & 0.7\\ \hline
    \end{tabular}
    \end{center}
  \captionof{table}{Coverage of $L_2$ credible sets based on 100 runs of the algorithm.Increasing sample size is considered from $N=500$ to $N=20000$, while the number of machines are kept fixed at $m=100$. Both methods and both the Mat\'ern kernel and the rescaled squared exponential covariance kernels are considered. \label{table:coverage}}
\end{table}

\section{Discussion}\label{sec:discussion}
In this paper, we have shown that distributed methods can be applied in the context of Gaussian Process regression and give accurate results in terms of recovery and uncertainty quantification. Although a naive averaging of the local posteriors will fail to capture the true functional parameter, there exist techniques obtaining a global posterior distribution which has similar asymptotic  behaviour as the non-distributed posterior distribution. We demonstrate through various examples (including both polynomially and exponentially decaying eigenvalues for the covariance kernel)  that the aggregated posterior distribution can achieve optimal minimax contraction rates and good frequentist coverage.

One of the main contributions of our paper is that we do not need to assume that the true functional parameter belongs to the Reproducing Kernel Hilbert Space (RKHS) corresponding to the considered Gaussian Process prior, which is a typical assumption in the literature. This way our results are less restrictive and can be applied for a larger class of functions and priors. For instance squared exponential covariance kernels contain analytic functions in their RKHS, hence assuming that the truth belongs to that space would substantially reduce the applicability of the method. Also, in case of Mat\'ern kernels by relaxing this assumption we do not have to introduce an (artificial) rescaling factor which is needed otherwise as the regularity of the Mat\'ern kernel can't be chosen to match the regularity of the truth.

The optimal choice of the tuning hyper-parameter in the covariance kernel depends on the regularity of the underlying function, which is typically unknown in practice. In the non-distributed setting various adaptive techniques were proposed to solve this problem, including hierarchical and empirical Bayes methods. However, in the distributed setting standard approaches based on the (marginal) likelihood fail, as it was demonstrated in the context of the Gaussian white noise model, see \cite{szabo:vzanten:19}. An open and interesting line of research is to understand whether adaptation is possible at all in the distributed regression framework \eqref{sec: GP fram} and if yes to provide method achieving it.

\section{Proofs of the main results}\label{sec: proofs}
\subsection{Kernel Ridge Regression in non-distributed setting}\label{sec:nondistributed:KRR}
Let us first consider the non-distributed case, i.e. take $m=1$. We introduce some notations and recall standard results for the kernel ridge regression method. The posterior mean $\hat{f}_N$ coincides with the kernel ridge regression (KRR) estimator
\begin{align}\label{eq: KRR}
\hat{f}_N=\hat{f}_{KRR}=\arg\min_{f\in\mathcal{H}}[-\ell_N(f)], ~ -\ell_N(f):=\sum_{i=1}^N(Y_i-f(X_i))^2+\sigma^2\|f\|_{\mathcal{H}}^2,
\end{align}
where the RKHS $\mathcal{H}$ corresponds to the prior covariance kernel $K$, see Chapter 6 in \cite{rasmussen:williams:2006}. The objective function of the KRR is composed of the average squared-error loss and an RKHS penalty term. In view of the representer theorem for RKHSs, the solution to \eqref{eq: KRR} is a linear combination of kernel functions, which renders it equivalent to a quadratic program. 

By the reproducing property, all functions $f$ in the RKHS $\mathcal{H}$ can be evaluated as $f(X_i)=\langle f,K_{X_i}\rangle_{\mathcal{H}}$ with $K_{X_i}=K(X_i,\cdot)$, and $\|f\|_{\mathcal{H}}^2=\langle f,f\rangle_{\mathcal{H}}$. The corresponding log-likelihood function takes the form (up to an additive constant term)
$$-\ell_N(f):=\sum_{i=1}^N(Y_i-\langle f,K_{X_i}\rangle_{\mathcal{H}})^2+\sigma^2\langle f,f\rangle_{\mathcal{H}}.$$
Performing a Fr\'{e}chet derivation on $\ell_N:(\mathcal{H},\langle \cdot,\cdot\rangle_{\mathcal{H}})\to\mathbb{R}$ with respect to $f$, one can obtain the score function. By multiplying the score function with $1/(2N)$ we arrive at the function $\hat{S}_N:\mathcal{H}\to\mathcal{H}$ given as
\begin{align}\label{eq: score KRR}
\hat{S}_N(f)=\frac{1}{N}\Bigg[\sum_{i=1}^N(Y_i-f(X_i))K_{X_i}-\sigma^2f\Bigg].
\end{align}
For simplicity we refer to $\hat{S}_{N}(f)$ as the score function from now on and note that the \textit{KRR} estimate $\hat{f}_N=\hat{f}_{KRR}$ then verifies 
\begin{align*}
\hat{S}_N(\hat{f}_N)=0.
\end{align*}
 Define also $S_N(f):=E_0 \hat{S}_N(f)$ to be the population version of the score function, i.e. 
\begin{align}\label{eq: pop score KRR}
S_N(f)=\int_{\mathcal{X}}(f_0(x)-f(x))K_{x}dx-\frac{\sigma^2}{N}f=F(f_0-f)-\frac{\sigma^2}{N}f,
\end{align}
where the operator $F:L_2(\mathcal{X})\to\mathcal{H}$ is a convolution with the kernel $K$, i.e. $F(g)=\int g(x)K_xdx$. Considering $g=\sum_{j\in\mathbb{N}^d}g_j\psi_j$, a straightforward calculation gives $F(g)=\sum_{j\in\mathbb{N}^d}\mu_jg_j\psi_j$. We can then rewrite $S_N(f)$ as
\begin{align}
S_N(f)=\sum_{j\in\mathbb{N}^d}\Big(\mu_jf_{0,j}-\frac{\sigma^2+\mu_jN}{N}f_j\Big)\psi_j,\label{hulp:new1}
\end{align}
which leads immediately to a solution of $S_N(f)=0$ with $f_j=\nu_{j}f_{0,j}$ where $\nu_{j}=\nu_{N,j}=\frac{\mu_jN}{\sigma^2+\mu_jN}$.

Let us define another operator $\tilde{F}:L_2(\mathcal{X})\to\tilde{\mathcal{H}}$, with $\tilde{\mathcal{H}}$ denoting the Hilbert space with inner product $\langle f,g \rangle_{\tilde{\mathcal{H}}}=\sum_{j\in\mathbb{N}} \nu_j^{-2} f_jg_j$,  as $\tilde{F}(g)=\sum_{j\in\mathbb{N}^d}\nu_{j}g_j\psi_j$ (we omit the dependence on $N$ in the notation). Note that both operators $F$ and $\tilde{F}$ are bijective and linear, which allows us to rewrite (\ref{eq: pop score KRR}) as
$$S_N(f)=F(f_0)-F \circ \tilde{F}^{-1}(f)=F(f_0-\tilde{F}^{-1}(f)).$$
Hence, using the notation $\Delta\hat{f}_{N}=\hat{f}_N-\tilde{F}(f_0)$ we get
\begin{align}\label{eq: help:delta}
\Delta\hat{f}_{N}=-\tilde{F}\circ F^{-1}\circ S_N(\hat{f}_N).
\end{align}
It will also be useful to define the operator $\tilde{P}=\mathrm{id}-\tilde{F}$, where $\mathrm{id}$ denotes the identity operator on $L_2(\mathcal{X})$. Also note that $S_N(\tilde{F}(f_0))=0$.

Table \ref{tb:notations} provides a summary of the key above notations in order to help the reader find a way in the proofs. 
\begin{table}[h]
\caption{Notation references}
\begin{tabular}{l|l}
\textbf{Symbol} & \textbf{Definition}\\
\hline
$\mathbb{D}_N$ & Data, $\{(Y_i,X_i)_{i=1}^N\}$.\\
$f_0$ & True function.\\
$\epsilon_i$ & Gaussian error, $\epsilon_i=Y_i-f_0(X_i)\sim\mathcal{N}(0,\sigma^2)$.\\
$\hat{f}_N$ & posterior mean function, $E_X[f|\mathbb{D}_N]$, equal to the KRR solution.\\
&$\hat{f}_N=\arg\min_{f\in\mathcal{H}}\Big[N^{-1}\sum_{i=1}^N(Y_i-f(X_i))^2+N^{-1}\sigma^2\|f\|_{\mathcal{H}}^2\Big]$.\\
$F$ & Convolution with kernel $K$, $F(g)=\sum_{j\in\mathbb{N}^d}\mu_j g_j\psi_j$.\\
$F^{-1}$ & Inverse of $F$, $F^{-1}(g)=\sum_{j\in\mathbb{N}^d}(g_j/\mu_j)\psi_j$.\\
$\{\nu_j\}_{j\in\mathbb{N}^d}$ & Eigenvalues of the equivalent kernel $\nu_j=\mu_j N/(\sigma^2+\mu_j N)$.\\
$\tilde{F}$ & Convolution with the equivalent kernel $\tilde{F}(g)=\sum_{j\in\mathbb{N}^d}\nu_j g_j\psi_j$.\\
$\tilde{F}^{-1}$ & Inverse of $\tilde{F}$, $\tilde{F}^{-1}(g)=\sum_{j\in\mathbb{N}^d}(g_j/\nu_j)\psi_j$.\\
$\tilde{P}$ & $\tilde{P}=\mathrm{id}-\tilde{F}$.\\
$\hat{S}_N$ & Sample score function, $\hat{S}_N(f)=N^{-1}[\sum_{i=1}^N(Y_i-f(X_i))K_{X_i}-\sigma^2f]$.\\
$S_N$ & Population score function, $S_N(f)=F(f_0-\tilde{F}^{-1}(f))$.
\end{tabular}
\label{tb:notations}
\end{table}

\subsection{Kernel Ridge Regression in distributed setting}

In the distributed setting (both in Methods I and II), accordingly, the $k$th local sample and population score functions are given (up to constant multipliers) by
\begin{align}
\hat{S}_{n}^{(k)}(f)&=\frac{1}{n}\Big[ \sum_{i=1}^{n} (Y_i^{(k)}-f(X_{i}^{(k)}))K_{X_i^{(k)}}-m^{-1}\sigma^2f \Big],\nonumber\\
S_n^{(k)}(f)&=   \int_{\mathcal{X}}(f_0(x)-f(x))K_{x}dx-\frac{\sigma^2}{nm }f=S_N(f),\label{eq: distr:score}
\end{align}
respectively.  Analogously to (\ref{eq: score KRR}), every local KRR estimate satisfies $\hat{S}_{n}^{(k)}(\hat{f}_{n}^{(k)})=0$. In view of $S_n^{(k)}=S_N$ we have $S_n^{(k)}\big(\tilde{F}(f_0)\big)=0$, hence  for each machine, let $\Delta\hat{f}_{n}^{(k)}=\hat{f}^{(k)}_{n}-\tilde{F}(f_0)$ denote the difference between the empirical and the population minimizer of the KRR.


\subsection{Proof of Theorem \ref{thm: contr:general}}
\label{sec: contr:general}
In the proof we use ideas from the proof of Theorem 2.1 of \cite{bhattacharya:pati:yun:17}. The main differences between their and our results are that we are considering (various) distributed Bayesian methods (not just the standard posterior with $m=1$) and that we extend the results to general Gaussian process priors (including kernel with polynomially decaying and exponentially decaying eigenvalues), while the proof \cite{bhattacharya:pati:yun:17} only covered the rescaled version of the kernel with polynomially decaying eigenvalues and scaling factor depending on the sample size. More specifically we do not require that the true function belongs to the RKHS of the GP prior, which substantially extends the applicability of our results. Finally in our analysis we consider the multivariate $d$-dimensional case, work with $L_2$-norm and consider Sobolev type of regularity classes rather than $L_{\infty}$ norm and hyper-rectangles induced by the series decomposition with respect to the eigenbasis $\boldsymbol\psi_j$. These extensions and conceptual differences required substantially different proof techniques than in \cite{bhattacharya:pati:yun:17}.

First note that in view of the inequality $(a+b)^2\leq 2a^2+2b^2$, we get
$$E_0\|\hat{f}_{n,m}-f_0\|_2^2\leq 2\|f_0-\tilde{F}(f_0)\|_2^2+2E_0\|\hat{f}_{n,m}-\tilde{F}(f_0)\|_2^2,$$
where $\hat{f}_{n,m}$ is the mean of the global posterior $\Pi^{\dagger}_{n,m}(.|\mathbb{D}_N)$ obtained with either Method $I$ or $II$. Then we show in Section \ref{sec: proof:bias} that for $f_0\in \Theta^{\beta}(B)$
\begin{align}
\label{eq: UB:bias}
E_0\|\hat{f}_{n,m}-\tilde{F}(f_0)\|_2^2\lesssim\Big(\frac{1}{N}\sum_{j\in\mathbb{N}^d}\nu_j^2\Big)\big(\|\tilde{P}(f_0)\|^2+\sigma^2\big)+ \delta_{N},
\end{align}
 where
\begin{align*}
\delta_{N}= \inf \Big\{ N\sum_{j\in\mathbb{N}^d}\nu_j^2 \sum_{\ell\in \mathcal{I}^c}\mu_\ell:\, \mathcal{I}\subset\mathbb{N}^d,\, |\mathcal{I}|\leq \frac{n}{m}(\sum_{j\in\mathbb{N}^d}\nu_j^2)^{-1}\Big\},
\end{align*}
concluding the proof of the first statement.

For the contraction rate note that by using Markov's and triangle inequalities we get
\begin{align*}
E_0\Pi_{n,m}^{\dagger}\Big(f:\|f-f_0\|_2 \geq M_N\epsilon_N|\mathbb{D}_N\Big)\leq 2\frac{E_0E_{n,m}^{\dagger}[\|f-\hat{f}_{n,m}\|_2^2|\mathbb{D}_N]+E_0\|\hat{f}_{n,m}-f_0\|_2^2}{M_N^2\epsilon_N^2}.
\end{align*}
Therefore it is sufficient to show that
\begin{align*}
E_0E_{n,m}^{\dagger}[\|f-\hat{f}_{n,m}\|_2^2|\mathbb{D}_N]=O(\frac{\sigma^2}{N}\sum_{j}\nu_j).
\end{align*}
In view of Fubini's theorem the expected squared $L_2$-norm of the process $f-\hat{f}_{n,m}|\mathbb{D}_N$ is the integral of the aggregated posterior variance of $f(x)$ over $\mathcal{X}$,
$$E_{n,m}^{\dagger}[\|f-\hat{f}_{n,m}\|_2^2|\mathbb{D}_N]=\int_{\mathcal{X}}\Var_{n,m}^{\dagger}[f(x)|\mathbb{D}_N]dx.$$
In the non-distributed setting, the posterior variance only depends on the design matrix $\mathbb{X}$. The expectation of this integral is known as the \textit{learning curve} in Chapter 7 of \cite{rasmussen:williams:2006}. In Section \ref{sec: proof:variance} we prove that
\begin{align}\label{eq: asymp post var}
E_0 \int_{\mathcal{X}} \Var_{n,m}^{\dagger}(f(x)|\mathbb{D}_N)dx\asymp \sigma^2\sum_{j\in\mathbb{N}^d}\frac{\mu_j}{\sigma^2+N\mu_j}=\frac{\sigma^2}{N}\sum_{j\in\mathbb{N}^d}\nu_j,
\end{align}
concluding the proof of the statement.

\subsubsection{Proof of \eqref{eq: UB:bias}}\label{sec: proof:bias}

First note, that in view of the inequality $(a+b)^2\leq 2a^2+2b^2$,
\begin{align*}
 \|\Delta\hat{f}_{n}^{(k)} \|_2^2&\leq 2  \Big\|\Delta\hat{f}_{n}^{(k)}-\tilde{F} \circ F^{-1}\circ \hat{S}^{(k)}_{n}(\tilde{F}(f_{0}))\Big\|_2^2
 + 2\Big\| \tilde{F} \circ F^{-1}\circ \hat{S}^{(k)}_{n}(\tilde{F}(f_{0}))\Big\|_2^2.
\end{align*}
Then we show below that
\begin{align}
 E_0\Big\|\Delta\hat{f}_{n}^{(k)}-\tilde{F} \circ F^{-1}\circ \hat{S}^{(k)}_{n}(\tilde{F}(f_{0}))\Big\|_2^2\lesssim \frac{1}{m} E_0\|\Delta\hat{f}_{n}^{(k)}\|_2^2+\delta_{N},\label{eq: hulp1:old}
\end{align}
which together with the preceding display implies
\begin{align*}
E_0\|\Delta\hat{f}_{n}^{(k)} \|_2^2\leq (2+o(1))\Big(E_0 \Big\|\tilde{F} \circ F^{-1}\circ \hat{S}^{(k)}_{n}(\tilde{F}(f_{0}))\Big\|_2^2+C\delta_{N}\Big).
\end{align*}
By combining the preceding two displays we arrive at
\begin{align*}
&E_0\Big\|\Delta\hat{f}_{n}^{(k)}-\tilde{F} \circ F^{-1}\circ \hat{S}^{(k)}_{n}(\tilde{F}(f_{0}))\Big\|_2^2\nonumber\\
&\qquad\qquad\lesssim \frac{1}{m}E_0   \Big\|\tilde{F} \circ F^{-1}\circ \hat{S}^{(k)}_{n}(\tilde{F}(f_{0}))\Big\|_2^2+\delta_{N}.
\end{align*}

For the aggregated estimator we get that
\begin{align*}
\|\Delta\hat{f}_{n,m} \|_2^2&\lesssim 
 \Big\|\Delta\hat{f}_{n,m}-\frac{1}{m}\sum_{k=1}^{m}\tilde{F} \circ F^{-1}\circ \hat{S}^{(k)}_{n}(\tilde{F}(f_{0}))\Big\|_2^2\\
&\qquad + \Big\|\frac{1}{m}\sum_{k=1}^m \tilde{F} \circ F^{-1}\circ \hat{S}^{(k)}_{n}(\tilde{F}(f_{0}))\Big\|_2^2.
\end{align*}
Then in view of the preceding display, the independence of the data across machines and $E_0 \big(\tilde{F} \circ F^{-1}\circ \hat{S}^{(k)}_{n}(\tilde{F}(f_{0}))\big)=0$ we get that
\begin{align*}
E_0\|\Delta\hat{f}_{n,m} \|_2^2\lesssim \frac{1}{m}E_0 \Big\|\tilde{F} \circ F^{-1}\circ \hat{S}^{(k)}_{n}(\tilde{F}(f_{0}))\Big\|_2^2+\delta_{N}.
\end{align*}
Finally we verify below that
\begin{align}
E_0\Big\|\tilde{F} \circ F^{-1}\circ \hat{S}^{(k)}_{n}(\tilde{F}(f_0))\Big\|_2^2\lesssim
\Big(\frac{1}{n}\sum_{j\in\mathbb{N}^d}\nu_j^2\Big)\big(\|\tilde{P}(f_0)\|_2^2+\sigma^2\big),
\label{eq: hulp2:old}
\end{align}
which together with $\|\tilde{P}(f_0)\|_2^2\leq \|f_0\|_2^2\leq B^2$
provides us \eqref{eq: UB:bias}.\\


\noindent\textbf{Proof of \eqref{eq: hulp1:old}:} First note that  the identity $\Delta\hat{f}_{n}^{(k)}=-\tilde{F} \circ  F^{-1}\circ S_n^{(k)}(\hat{f}^{(k)}_{n})$ follows from assertions \eqref{eq: help:delta} and \eqref{eq: distr:score}. This implies together with the properties of $\hat{S}^{(k)}_{n}$ and $S_n^{(k)}$, that
\begin{align}
\big(\hat{S}^{(k)}_{n}(\hat{f}^{(k)}_{n})-S_n^{(k)}(\hat{f}^{(k)}_{n})\big)&-\big(\hat{S}^{(k)}_{n}(\tilde{F}(f_0))-S_n^{(k)}(\tilde{F}(f_0))\big)\nonumber\\
&=F\circ\tilde{F}^{-1}(\Delta\hat{f}_{n}^{(k)})-\hat{S}^{(k)}_{n}(\tilde{F}(f_0)).\label{eq: help:equiv}
\end{align}
On the other hand, in view of \eqref{eq: distr:score},
$$\hat{S}^{(k)}_{n}(f)-S_n^{(k)}(f)=\frac{1}{n}\sum_{i=1}^{n}\big(Y_i^{(k)}-f(X_i^{(k)})\big)K_{X_i^{(k)}}-\int_{\mathcal{X}}\big(f_0(x)-f(x)\big)K_xdx$$
for all functions $f\in \mathcal{H}$. Therefore, by applying the preceding display twice with $f=\hat{f}_{n}^{(k)}$ and $f=\tilde{F}(f_0)$, we get that
\begin{align*}
\big(\hat{S}^{(k)}_{n}(\hat{f}^{(k)}_{n})&-S_n^{(k)}(\hat{f}^{(k)}_{n})\big)-\big(\hat{S}^{(k)}_{n}(\tilde{F}(f_0))-S_n^{(k)}(\tilde{F}(f_0))\big)\\
&=-\frac{1}{n}\sum_{i=1}^{n}\Delta\hat{f}_{n}^{(k)}(X_i^{(k)})K_{X_i^{(k)}}+\int_{\mathcal{X}}\Delta\hat{f}_{n}^{(k)}(x)K_{x}dx.
\end{align*}
Combining assertion \eqref{eq: help:equiv} with the preceding display and then using Lemma \ref{lem:bound:random} (with $\hat{g}=\Delta\hat{f}_{n}^{(k)}$, satisfying the boundedness assumption, see Lemma \ref{lem:help:bounded}) together with Lemma \ref{lem:largeJ},  we get for arbitrary index set $\mathcal{I}\subset\mathbb{N}^d$ that
\begin{align*}
&E_0\Big\|\Delta\hat{f}_{n}^{(k)} - \tilde{F} \circ F^{-1}\circ S^{(k)}_{n}(\tilde{F}(f_0))\Big\|_2^2\nonumber\\
&\qquad= E_0\Big\|(\tilde{F}\circ F^{-1})\Big(\frac{1}{n}\sum_{i=1}^{n}\Delta\hat{f}_{n}^{(k)}(X_i^{(k)})K_{X_i^{(k)}}-\int_{\mathcal{X}}\Delta\hat{f}_{n}^{(k)}(x)K_{x}dx\Big)\Big\|_2^2\nonumber\\
&\qquad\lesssim \frac{ |\mathcal{I}|\sum_{j\in\mathbb{N}^d}\nu_j^2}{n}E_0\|\Delta\hat{f}_{n}^{(k)}\|_2^2+ 
N\sum_{j\in\mathbb{N}^d}\nu_j^2\sum_{\ell \in \mathcal{I}^c}\mu_\ell.
\end{align*}
Taking the minimum over $|\mathcal{I}|\leq \frac{n}{m}(\sum_{j\in\mathbb{N}^d}\nu_j^2)^{-1}$ we get that
\begin{align}
E_0\Big\|\Delta\hat{f}_{n}^{(k)} - \tilde{F} \circ F^{-1}\circ S^{(k)}_{n}(\tilde{F}(f_0))\Big\|_2^2\lesssim \frac{1}{m}E_0\|\Delta\hat{f}_{n}^{(k)}\|_2^2+N\delta_{N}\label{eq: hulp4}
\end{align}
concluding the proof of \eqref{eq: hulp1:old}.\\

\noindent\textbf{Proof of \eqref{eq: hulp2:old}}.  In view of the linearity of the operator $\tilde{F}\circ F^{-1}$, the inequality $\|f+g\|_2^2\leq 2\|f\|_2^2+2\|g\|_2^2$, and
\begin{align*}
\hat{S}^{(k)}_{n}(\tilde{F}(f_0))&=\frac{1}{n}\sum_{i=1}^{n}(Y_i^{(k)}-f_0(X_i^{(k)}))K_{X_i^{(k)}}\\&+\frac{1}{n}\sum_{i=1}^{n}\tilde{P}(f_0)(X_i^{(k)})K_{X_i^{(k)}}-\frac{\sigma^2}{mn}\tilde{F}(f_0),
\end{align*}
the left hand side of \eqref{eq: hulp2:old} can be bounded from above as
\begin{align*}
E_0&\Big\|\tilde{F} \circ F^{-1}\Big(\hat{S}^{(k)}_{n}(\tilde{F}(f_0))- S_n^{(k)}(\tilde{F}(f_0))\Big) \Big\|_2^2\\
&\quad \leq  2 E_0\Big\|\tilde{F} \circ F^{-1}\Big(\frac{1}{n} \sum_{i=1}^{n}\tilde{P}(f_0)(X_i^{(k)})K_{X_i^{(k)}}-E_X[\tilde{P}(f_0)(X)K_{X}]\Big)\Big\|_2^2\\
&\qquad\qquad+2E_0\Big\|\tilde{F} \circ F^{-1}\Big(\frac{1}{n}\sum_{i=1}^{n}\epsilon_i^{(k)}K_{X_i^{(k)}}\Big)\Big\|_2^2\\
&\quad=:(T_1+T_2).
\end{align*}
We deal with terms $T_1$ and $T_2$ separately. In view of Lemma \ref{lem: bound} (with $g=\tilde{P}(f_0)$)
$$T_1\leq \frac{2C}{n}\sum_{j\in\mathbb{N}^d}\nu_j^2\|\tilde{P}(f_0)\|_2^2,$$
for some $C>0$. Since the operator $\tilde{F}\circ F^{-1}$ is linear, we get that
\begin{align*}
T_2&=\frac{2}{n^2}\sum_{i=1}^{n}E_0\big( (\epsilon_i^{(k)})^2\|\tilde{F}\circ F^{-1}(K_{X_i^{(k)}})\|_2^2\big)\\
&\qquad+\frac{4}{n^2}\sum_{1\leq i<\ell\leq n}E_0\Big(\epsilon_i^{(k)}\epsilon_\ell^{(k)}\tilde{F}\circ F^{-1}(\langle K_{X_i^{(k)}},K_{X_\ell^{(k)}}\rangle_2)\Big)\\
&=\frac{2\sigma^2}{n}E_0\|\tilde{F} \circ F^{-1}(K_{X_1^{(k)}})\|_2^2=\frac{2\sigma^2}{n}\sum_{j\in\mathbb{N}^d}\nu_j^2,
\end{align*}
because the cross terms are equal to $0$ due to independence of the noise $\eps_i^{(k)}$, $i=1,...,n$, $k=1,...,m$.

\subsubsection{Proof of \eqref{eq: asymp post var}}\label{sec: proof:variance}
In this section we give upper bounds for the learning curves in case of both distributed methods.

\paragraph{Method I:} Let us denote by $\mu_j^{I}=m\mu_j$ the eigenvalues of the local covariance kernel. Then in view of 
Lemma \ref{lem: bound exp post var}, the expectations of the $m$ local posterior variances are all of the same order
$$E_0E_X\Var(f(X)|\mathbb{D}_{n}^{(k)})\asymp\sigma^2\sum_{j\in\mathbb{N}^d}\frac{\mu^I_j}{\sigma^2+n\mu^I_j}=\sigma^2\sum_{j\in\mathbb{N}^d}\frac{m\mu_j}{\sigma^2+N\mu_j}=\frac{\sigma^2}{n}\sum_{j\in\mathbb{N}^d}\nu_j.$$
 Since the variance of the global posterior distribution $\Pi^{I}_{n,m}(.|\mathbb{D}_N)$ satisfies the following equality 
$$\Var_{n,m}^I(f(x))=m^{-2}\sum_{k=1}^m\Var(f(x)|\mathbb{D}_{n}^{(k)}),$$
one can see that 
$$E_0E_X\Var_{n,m}^I(f(X))\asymp\frac{\sigma^2}{N}\sum_{j\in\mathbb{N}^d}\nu_j.$$

\paragraph{Method II:}  First note that $\mu_j^{II}=\mu_j$ the eigenvalues of the local covariance kernel. Note that the expectations of the $m$ local posterior variances are all of the same order
$$E_0E_X\Var(f(X)|\mathbb{D}_{n}^{(k)})\asymp\frac{\sigma^2}{m}\sum_{j\in\mathbb{N}^d}\frac{\mu^{II}_j}{\sigma^2/m+n\mu^{II}_j}=\frac{\sigma^2}{N}\sum_{j\in\mathbb{N}^d}\nu_j,$$
because the variance of the noise is $\sigma^2/m$ for each machine. The variance of the aggregated posterior distribution $\Pi^{II}_{n,m}(.|\mathbb{D}_N)$ satisfies
$$E_0E_X \Var_{n,m}^{II}(f(X)|\mathbb{D}_N)\asymp\frac{\sigma^2}{N}\sum_{j\in\mathbb{N}^d}\nu_j$$
because we know that
$$\Var_{n,m}^{II}(f(X)|\mathbb{D}_N)=m^{-1}\sum_{k=1}^m\Var(f(X)|\mathbb{D}_{n}^{(k)})$$
proving assertion \eqref{eq: asymp post var}.

\subsection{Proof of Theorem \ref{thm:counter}}\label{sec:thm:counter}

The proof follows similar lines of reasoning as Theorem \ref{thm: contr:general}, where we provided general upper bounds for the contraction rate of the distributed posterior.

First we prove \eqref{eq: bad_estimate}. For the naive averaging method the local sample and population score functions coincide to the non-distributed case given in Section \ref{sec:nondistributed:KRR} with sample size $n$, i.e.
\begin{align*}
\hat{S}^{*(k)}_{n}(f)=\frac{1}{n}\Bigg[\sum_{i=1}^{n}(Y_i^{(k)}-f(X_i^{(k)}))K_{X_i^{(k)}}-f\Bigg],\\
S_n^{*(k)}(f)=\int_{\mathcal{X}}(f_0(x)-f(x))K_{x}dx-\frac{\sigma^2}{n}f=F(f_0-f)-\frac{1}{n}f.
\end{align*}
Not that the solution of the equation $S_n^{*(k)}(f)=0$ is given by the coefficients $f_j=\nu^*_jf_{0,j}$, with $\nu^*_j=\frac{\mu_jn}{1+\mu_jn}$, $j\in\mathbb{N}^d$.

Then using the inequality $a^2\geq (a-b)^2/2-b^2$ one can obtain that
$$E_0\|\hat{f}_{n,m}^*-f_0\|_2^2\geq \frac{1}{2}\|f_0-\tilde{F}^*(f_0)\|_2^2-E_0\|\hat{f}_{n,m}^*-\tilde{F}^*(f_0)\|_2^2,$$
where $\tilde{F}^*(g)=\sum_{j\in\mathbb{N}}\nu_j^* g_j\psi_j$ and $\hat{f}_{n,m}^*$ is the mean of the global posterior $\Pi^{*}_{n,m}(.|\mathbb{D}_N)$ obtained with the naive averaging method.

First note that
\begin{align}
\|f_0-\tilde{F}^*(f_0)\|_2^2&=\sum_{j=1}^\infty\frac{m}{m+\mu_j^{*}n }f_{0,j}^2\geq \frac{c_L}{2} \sum_{(n/\sigma^2)^{1/(1+2\beta)}\leq j} j^{-1-2\beta}(\log j)^{-2}\nonumber\\
& \geq c_0 n^{-2\beta/(1+2\beta)}  (\log n)^{-2},\label{eq:counter:LB:approx}
\end{align}
for some small enough $c_0>0$. We conclude the proof of \eqref{eq: bad_estimate} by showing below that $E_0\|\hat{f}_{n,m}-\tilde{F}(f_0)\|_2^2= o(n^{-2\beta/(2\beta+1)}  (\log n)^{-2})$. 

Similarly to \eqref{eq: UB:bias} we can derive (by replacing $\tilde{F}$  and $\nu$ with $\tilde{F}^*$  and $\nu^*$, respectively) that
\begin{align*}
E_0\|\hat{f}_{n,m}-\tilde{F}^*(f_0)\|_2^2\lesssim\Big(\frac{1}{N}\sum_{j=1}^\infty(\nu_j^{*})^2\Big)\big(\|\tilde{P}^*(f_0)\|_2^2+\sigma^2\big)+ \delta_{N}^*,
\end{align*}
where $\delta_{N}^*= N\sum_{j=1}^\infty(\nu_j^*)^2 \sum_{\ell=I}^{\infty}\mu_\ell$, with $I= n/\big(m\sum_{j=1}^{\infty}(\nu_j^*)^2\big)$.
Note that $\|\tilde{P}^*(f_0)\|_2^2=O(1)$ and in view of Lemmas \ref{lem: bound eigen poly} and \ref{lem: bound eigen expo}, $\sum_{j=1}^\infty(\nu_j^*)^2\asymp n^{1/(1+2\beta)}$; hence 
$$I\asymp n^{2\beta/(1+2\beta)}/m.$$
Therefore the first term on the right hand side of the preceding display is $O(n^{-2\beta/(1+2\beta)}/m)$ and
$$\delta_N^*\lesssim n^{1/(1+2\beta)}NI^{-2\beta}\asymp N^{2-2\beta}m^{-1+4\beta}=o(n^{-2\beta/(1+2\beta)}  (\log n)^{-2}),$$
where the last step holds for large enough choice of $\beta$ and not to large choice of $m$. For instance taking $\beta\geq 2$ and $m=o(N^{1/(2+2\beta)})$ we get that
$$\delta_N^* n^{2\beta/(2\beta+1)}\lesssim N^{-1/6}=o(\log^{-2} n).$$

It remained to deal with \eqref{eq: bad_contraction}.  First note that by the computations above combined with Markov's inequality there exists a sequence $\rho_n\rightarrow 0$ such that
\begin{align*}
P_0( \|\hat{f}_{n,m}^*-\tilde{F}^*(f_0)\|_2\geq \rho_n n^{-\beta/(1+2\beta)}(\log n)^{-1})\rightarrow 0.
\end{align*}
Then by triangle inequality, \eqref{eq:counter:LB:approx} and Markov's inequality we get for $c<c_0$ that
\begin{align*}
E_0\Pi_{n,m}^*&\Big(f:\|f-f_{0}\|_2\leq c n^{-\beta/(1+2\beta)}(\log n)^{-1}|\mathbb{D}_N\Big)\\
&\leq  E_0 \Pi_{n,m}^*\Big(  \|f_0-\tilde{F}^*(f_0)\|_2 - c n^{-\frac{\beta}{1+2\beta}}(\log n)^{-1}\\
&\qquad\qquad -\|\hat{f}_{n,m}^*-\tilde{F}^*(f_0)\|_2 \leq \|f-\hat{f}_{n,m}^*\|_2|\mathbb{D}_N\Big)\\
&\leq  E_0 \Pi_{n,m}^*\Big(  (c_0-c-\rho_n) n^{-\beta/(1+2\beta)}(\log n)^{-1} \leq \|f-\hat{f}_{n,m}^*\|_2|\mathbb{D}_N\Big)+o(1)\\
&\lesssim n^{2\beta/(2\beta+1)}(\log n)^2 E_0E^*_{n,m}\|f-\hat{f}_{n,m}^*\|^2_2.
\end{align*}

We conclude the proof by noting that
$$E_0E_X\Var\left(f(X)|\mathbb{D}^{(k)}_n\right)=\sigma^2\sum_{j=1}^{\infty}\frac{\mu_j}{\sigma^2+n\mu_j}=\frac{\sigma^2}{n}\sum_{j=1}^{\infty}\nu^*_j,$$
for all $k\in\{1,...,m\}$, hence
$$E_0 E_{n,m}^{*}\|f-\hat{f}_{n,m}^*\|_2^2=\frac{1}{m^2}\sum_{k=1}^mE_0E_X\Var\left(f(X)|\mathbb{D}^{(k)}_n\right)=\frac{\sigma^2}{N}\sum_{j=1}^{\infty}\nu^*_j\lesssim \frac{\sigma^2}{m}n^{-2\beta/(2\beta+1)}.$$

\subsection{Proof of Theorem \ref{th: unc quant}}
We first consider the non-distributed case $m=1$ for clearer presentation and then extend our results to the distributed setting.
\subsubsection{Non-distributed setting}\label{sec:nondistr:cov}
\paragraph{Connection to KRR}

Similarly to the posterior mean, the posterior covariance function $\hat{C}_N$ can be given as
$$\hat{C}_N(x,x')=K(x,x')-\hat{K}_N(x,x'),$$
where $\hat{K}_N(x,\cdot)=K(\cdot,\mathbb{X})[K(\mathbb{X},\mathbb{X})+\sigma^2I_N]^{-1}K(\mathbb{X},x)$, or equivalently
\begin{align}\label{eq: cov:KRR}
 \hat{K}_{x,N}=\hat{K}_N(x,\cdot)=\arg\min_{g\in\mathcal{H}}\Big[\frac{1}{N}\sum_{i=1}^N(K(x,X_i)-g(X_i))^2+\frac{\sigma^2}{N}\|g\|_{\mathcal{H}}^2\Big],
\end{align}
see assertion (8) of \cite{bhattacharya:pati:yun:17}.

Then by taking the Frechet derivative of the expression on the right hand side we arrive to the (adjusted) score function and its expected value
\begin{align*}
\hat{S}_{K_x,N}(g)=N^{-1}\Big(\sum_{i=1}^N \big(K_x(X_i)-g(X_i)\big)K_{X_i}-\sigma^2 g \Big),\\
S_{K_x,N}(g)=E\hat{S}_{K_x,N}(g)= \int_{\mathcal{X}}\big(K_x(z)-g(z) \big)K_z dz-\frac{\sigma^2}{N}g.
\end{align*}
Then similarly to the posterior mean in Section \ref{sec:nondistributed:KRR} the following assertions hold
\begin{align}
S_{K_x,N}(g)&=F(K_x)-F\circ \tilde{F}^{-1}(g)=F\big(K_x-\tilde{F}^{-1}(g)\big),\label{eq:cov:help1}\\
\Delta \hat{K}_{x,N}&= \hat{K}_{x,N}-\tilde{F}(K_x)=-\tilde{F}\circ F^{-1} \circ S_{K_x,N}( \hat{K}_{x,N}),\label{eq:cov:help2}\\
\hat{S}_{K_x,N}(\tilde{F}(K_x))&=\frac{1}{N}\Big( \sum_{i=1}^N \tilde{P}(K_x)(X_i)K_{X_i}-\sigma^2 \tilde{F}(K_x) \Big),\label{eq:cov:help3}\\
F\circ \tilde{F}^{-1}(\Delta \hat{K}_{x,N})&-\hat{S}_{K_x,N}( \tilde{F}(K_x))\nonumber \\
&= - \frac{1}{N} \sum_{i=1}^N \Delta \hat{K}_{x,N}(X_i)K_{X_i}+\int_{\mathcal{X}} \Delta \hat{K}_{x,N}(x')K_{x'} dx', \label{eq:cov:help4}
\end{align}
and note that $ \hat{K}_{x,N}$ and $\tilde{F}(K_x)$ are the zero points of the functions $\hat{S}_{K_x,N}$ and $S_{K_x,N}$, respectively.
\paragraph{Under-smoothing}

Following from the triangle inequality, to obtain frequentist coverage for the credible ball it is sufficient to show that for $L_n\rightarrow \infty$
$$P_0\Big(\|\tilde{P}(f_0)\|_2+\|\hat{f}_N-\tilde{F}(f_0)\|_2\leq L_N r_{N,\gamma}\Big)\rightarrow 1.$$
The preceding display is implied by assumption $\eqref{cond:biasUB}$ and assertions
\begin{align}
\label{eq: bound var - undersmth}
P_0\Big(\|\Delta\hat{f}_{N}\|_2^2\leq L_N\frac{\sigma^2}{N}\sum_{j\in\mathbb{N}^d}\nu_j\Big)\to 1,\\
\label{eq: bound radius - undersmth}
{P_0\Big(}r_{N,\gamma}^2\geq \frac{1}{2C_{\psi}^2} \frac{\sigma^2}{N}\sum_{j\in\mathbb{N}^d}\nu_j{\Big)\to 1},
\end{align}
where $\Delta\hat{f}_{N}:=\hat{f}_N-\tilde{F}(f_0)$, verified below.\\

\textbf{Proof of \eqref{eq: bound var - undersmth}:} In view of assertion \eqref{eq: UB:bias} with $m=1$ (and hence $n=N$) and Markov's inequality we get 
\begin{align}
P_0\Big(\|\Delta\hat{f}_{N}\|_2^2\geq L_N\frac{\sigma^2}{N}\sum_{j\in\mathbb{N}^d}\nu_j\Big)&\leq \frac{E_0 \|\Delta\hat{f}_{N}\|_2^2}{L_N\frac{\sigma^2}{N}\sum_{j\in\mathbb{N}^d}\nu_j}\nonumber\\
&\lesssim \frac{\Big(N^{-1}\sum_{j\in\mathbb{N}^d}\nu_j^2\Big)+  \delta_{N}}{L_N\sigma^2N^{-1}\sum_{j\in\mathbb{N}^d}\nu_j}\nonumber\\
&=O\Big( 1/L_N+ N\delta_{N}/\sum_{j\in\mathbb{N}^d}\nu_j  \Big)=o(1).\label{hulp:new:star}
\end{align}

\textbf{Proof of \eqref{eq: bound radius - undersmth}:} The radius $r_{N,\gamma}$ is defined, {conditionally on $\mathbb{X}$}, as $P(\|W_N\|_2^2\leq r_{N,\gamma}^2|\mathbb{X})=1-\gamma$, where $W_N$ is a centered GP with covariance kernel $\hat{C}_N$ given in \eqref{eq: postcov}. In view of Chebyshev's inequality
$$r_{N,\gamma}^2\geq E[\|W_N\|_2^2|\mathbb{X}]-(1-\gamma)^{-1/2}\Var(\|W_N\|_2^2|\mathbb{X})^{1/2}.$$
Using Fubini's theorem, the first term on the right hand side of the preceding display can be rewritten as
$$E[\|W_N\|_2^2|\mathbb{X}]=E_{\Pi}[\|f-\hat{f}_N\|_2^2|\mathbb{D}_N]=\int_{\mathcal{X}}\Var_{\Pi}[f(x)|\mathbb{D}_N]dx.$$
{The} integral on the right-hand side of the display, called the {\textit{generalization error}}, see Chapter 7 of \cite{rasmussen:williams:2006}, is asymptotically bounded from {below} almost surely by
\begin{equation}
\sigma^2\sum_{j\in\mathbb{N}^d}\frac{\mu_j}{\sigma^2+N\mu_j{\sup_{x\in\mathcal{X}} \boldsymbol{\psi}^2_j(x)}}{\geq}\frac{\sigma^2C_{\psi}^{-2}}{N}\sum_{j\in\mathbb{N}^d}\nu_j,\label{eq: hulp1}
\end{equation}
in view of assertion (12) of \cite{opper:vivarelli:1999} and Assumption \ref{ass: bound eigenf}. Furthermore, the variance of $\|W_N\|_2^2$, conditional on the design $\mathbb{X}$, is
$$\Var(\|W_N\|_2^2|\mathbb{X})=E[\|W_N\|_2^4|\mathbb{X}]-E^2[\|W_N\|_2^2|\mathbb{X}].$$
The first term on the right hand-side satisfies
\begin{align}
E[\|W_N\|_2^4|\mathbb{X}]&=E_{\Pi}[\|f-\hat{f}_N\|_2^4|\mathbb{D}_N]\\
&=\int\Bigg(\int_{\mathcal{X}}(f(x)-\hat{f}_N(x))^2dx\int_{\mathcal{X}}(f(x')-\hat{f}_N(x'))^2dx'\Bigg)\Pi(df|\mathbb{D}_N)\nonumber\\
&=\int_{\mathcal{X}}\int_{\mathcal{X}}\int(f(x)-\hat{f}_N(x))^2(f(x')-\hat{f}_N(x'))^2\Pi(df|\mathbb{D}_n)dxdx'\nonumber\\
&=\int_{\mathcal{X}}\int_{\mathcal{X}} \Var_{\Pi}[f(x)|\mathbb{D}_N]\Var_{\Pi}[f(x')|\mathbb{D}_N] +2\hat{C}_N(x,x')^2dx'dx\nonumber\\
&=\big(\int_{\mathcal{X}}\Var_{\Pi}[f(x)|\mathbb{D}_N]dx\big)^2+2\int_{\mathcal{X}}\|\hat{C}_N(x,.)\|_2^2dx\nonumber\\
&=E^2[\|W_N\|_2^2|\mathbb{X}]+2\int_{\mathcal{X}}\|\hat{C}_N(x,.)\|_2^2dx,\label{eq: var_term}
\end{align}
using Fubini's theorem and the reduction formula $EX_1^2X_2^2=Var(X_1)Var(X_2)+2Cov(X_1,X_2)^2$ for $X_1,X_2$ centered Gaussian random variables, see for instance page 189 of \cite{isserlis:1916}. Hence, again in view of Fubini's theorem,
\begin{align}
{E_0}\Var(\|W_N\|_2^2|\mathbb{X})=2\int_{\mathcal{X}}E_0\|\hat{C}_N(x,.)\|_2^2dx.\label{eq: help2}
\end{align}
Recall that the covariance function $\hat{C}_N(x,x')=K(x,x')-\hat{K}_N(x,x')$, where $\hat{K}_{x,N}=\hat{K}_N(x,.)$ is the solution to \eqref{eq: cov:KRR}. We show below that for all $x\in\mathcal{X}$
\begin{align}\label{eq: bound post_cov}
E_0\|\hat{C}_N(x,.)\|_2^2\lesssim  \|\tilde{P}(K_x)\|_2^2+\tilde\delta_{N},
\end{align}
for
$$\tilde\delta_{N}= \Big\{ (\sum_{j\in\mathbb{N}^d}\nu_j^2)^2 \sum_{\ell\in \mathcal{I}^c}\mu_\ell:\, \mathcal{I}\subset\mathbb{N}^d,\, |\mathcal{I}|=o(N/\sum_{j\in\mathbb{N}^d}\nu_j^2).\Big\}.$$
 In view of the definition of the linear operator $\tilde{P}$ and the eigenvalues $\nu_j$, $\mu_j$ we get
\begin{align}
\tilde{P}(K(x,x'))=\sum_{j\in\mathbb{N}^d}(1-\nu_j)\mu_j\boldsymbol{\psi}_j(x)\boldsymbol{\psi}_j(x')=\frac{\sigma^2}{N}\sum_{j\in\mathbb{N}^d}\nu_j\boldsymbol{\psi}_j(x)\boldsymbol{\psi}_j(x'),\label{eq:tilde:P:K}
\end{align}
for all $x,x'\in\mathcal{X}$. Then by combining the last three displays
\begin{align}
E_0\Var(\|W_N\|_2^2|\mathbb{X})&=2\int_{\mathcal{X}}E_0\|\hat{C}_N(x,.)\|_2^2dx\nonumber\\
&\lesssim \int_{\mathcal{X}}\|\tilde{P}(K(x,.))\|_2^2dx+\tilde\delta_{N}\nonumber\\
&=\Big(\frac{\sigma^2}{N}\Big)^2\int_{\mathcal{X}}\sum_{j\in\mathbb{N}^d}\nu_j^2\boldsymbol{\psi}_j(x)^2dx+\delta_{N}\frac{\sum_{j\in\mathbb{N}^d}\nu_j^2}{N}\nonumber\\
&=\Big(\frac{\sigma^2}{N}\Big)^2\sum_{j\in\mathbb{N}^d}\nu_j^2+ \delta_{N}\frac{\sum_{j\in\mathbb{N}^d}\nu_j^2}{N}.\label{eq: help3}
\end{align}
 Therefore, by Markov's inequality and Lemmas \ref{lem: bound eigen poly} and \ref{lem: bound eigen expo}, 
$$P_0\left(\Var\left(\|W_N\|_2^2|\mathbb{X}\right)^{1/2}\geq t \frac{\sigma^2}{N}\sum_{j\in\mathbb{N}^d}\nu_j\right)\lesssim t^{-2}\left(\frac{\sum_{j\in\mathbb{N}^d}\nu_j^2}{(\sum_{j\in\mathbb{N}^d}\nu_j)^2}+   \frac{N\delta_{N}\sum_{j\in\mathbb{N}^d}\nu_j^2}{(\sum_{j\in\mathbb{N}^d}\nu_j)^2} \right)\to 0$$
{for all $t> 0$. Hence by combining $\eqref{eq: hulp1}$ and the preceding display (with $t= (1-\gamma)^{1/2}C_{\psi}^{-2}/2$),
$${P_0\Big( E[\|W_N\|_2^2|\mathbb{X}]-(1-\gamma)^{-1/2}\Var(\|W_N\|_2^2|\mathbb{X})^{1/2}\geq (C_{\psi}^{-2}/2)\frac{\sigma^2}{N}\sum_{j\in\mathbb{N}^d}\nu_j\Big)\to 1}.$$

This implies that all the quantiles of $\|W_N\|_2^2$, {conditionally on $\mathbb{X}$}, are of the order $(\sigma^2/N)\sum_{j\in\mathbb{N}^d}\nu_j$ {with $P_0$-probability going to one}, including $r_{N,\gamma}^2$.\\

\textbf{Proof of \eqref{eq: bound post_cov}:} First note that by the inequality $(a+b)^2\leq 2a^2+2b^2$,
\begin{align*}
\|\hat{C}_N(x,.)\|_2^2\leq 2\|\tilde{P}(K_x)\|_2^2+ 2\|\Delta \hat{K}_{x,N}\|_2^2,
\end{align*}
where $\Delta \hat{K}_{x,N}=\hat{K}_{x,N}-\tilde{F}(K_x)$.

Next we give an upper bound for the second term of the preceding display similarly to Section \ref{sec: proof:bias}. First note  that
\begin{align*}
\|\Delta \hat{K}_{x,N}\|_2^2\lesssim  \| \Delta \hat{K}_{x,N}- \tilde{F}\circ F^{-1} \circ\hat{S}_{K_x,N}(\tilde{F}(K_{x}))  \|_2^2 +\|  \tilde{F}\circ F^{-1}\circ \hat{S}_{K_x,N}(\tilde{F}(K_{x}))  \|_2^2.
\end{align*}
Then by showing below that
\begin{align}
 E_0\Big\|\Delta\hat{K}_{x,N}-\tilde{F} \circ F^{-1}\circ \hat{S}_{K_x,N}(\tilde{F}(K_{x}))\Big\|_2^2\leq o( E_0\|\Delta\hat{K}_{x,N}\|_2^2)+\tilde\delta_{N},\label{eq: hulp:var1}
\end{align}
 we arrive at
\begin{align*}
 E_0\|\Delta \hat{K}_{x,N}\|_2^2&\lesssim  E_0 \|  \tilde{F}\circ F^{-1}\circ \hat{S}_{K_x,N}(\tilde{F}(K_{x}))  \|_2^2+\tilde\delta_{N}.
\end{align*}

Next, in view of \eqref{eq:cov:help3},
\begin{align*}
E_0\|  \tilde{F}\circ F^{-1}\circ &\hat{S}_{K_x,N}(\tilde{F}(K_{x}))  \|_2^2 
=E_0\Big\|  \tilde{F}\circ F^{-1}\Big(\hat{S}_{K_x,N}( \tilde{F}\big( K_{x}))- {S}_{K_x,N}( \tilde{F}\big( K_{x})) \Big) \Big \|_2^2\\ 
&= E_0\Big\|\tilde{F} \circ F^{-1}\Big(\frac{1}{N} \sum_{i=1}^{N}\tilde{P}(K_x)(X_i)K_{X_i}-E_X[\tilde{P}(K_x)(X)K_{X}]\Big)\Big\|_2^2\\
&\leq
\Big(\frac{1}{N}\sum_{j\in\mathbb{N}^d}\nu_j^2\Big)\|\tilde{P}(K_x)\|_2^2=o\left(\left\|\tilde{P}(K_x)\right\|_2^2\right),
\end{align*}
where the last line follows from Lemma \ref{lem: bound} with $g=\tilde{P}(K_x)$ (and $m=1$), concluding the proof of \eqref{eq: bound post_cov}.\\

\textbf{Proof of \eqref{eq: hulp:var1}:}\label{sec:proof: hulp:var1}
Similarly to \eqref{eq: hulp4}, by using assertion \eqref{eq:cov:help4}, Lemma \ref{lem:bound:random} (with $\hat{g}=\Delta\hat{K}_{x,N}$ and sample size $N$) and Lemma \ref{lem:cov:UB:help} (with $m=1$), we can show that for all $x\in\mathcal{X}$
\begin{align*}
&E_0\Big\|\Delta\hat{K}_{x,N}- \tilde{F} \circ F^{-1}\circ \hat{S}_{K_x,N}(\tilde{F}(K_{x}))\Big\|_2^2\nonumber\\
&\qquad= E_0\Big\|(\tilde{F}\circ F^{-1})\Big(\frac{1}{N}\sum_{i=1}^{N}\Delta\hat{K}_{x,N}(X_i)K_{X_i}-\int_{\mathcal{X}}\Delta\hat{K}_{x,N}(x')K_{x'}dx'\Big)\Big\|_2^2\nonumber\\
&\qquad\lesssim \frac{|\mathcal{I}|\sum_{j\in\mathbb{N}^d}\nu_j^2}{N}E_0\|\Delta\hat{K}_{x,N}\|_2^2+\Big(\sum_{j\in\mathbb{N}^d}\nu_j^2\Big)^2\sum_{\ell\in \mathcal{I}^c}\mu_\ell.
\end{align*}

Taking the infimum over $|\mathcal{I}|\leq o\big(N/\sum_{j\in\mathbb{N}^d}\nu_j^2)$ we get that the left hand side of the preceding display is bounded from above by $o(E_0\|\Delta\hat{K}_{x,N}\|_2^2)+\tilde\delta_{N}$, concluding the proof of the statement.

\paragraph{Over-smoothing}
By the definition of credible sets and using the triangle inequality, we get that
\begin{align*}
P_0\left(f_0\in\hat{B}_{n,m,\gamma}(L)\right)&\leq P_0\left(\left\|\tilde{P}(f_0)\right\|_2\leq\left\|\hat{f}_n-\tilde{F}(f_0)\right\|_2+Lr_{N,\gamma}\right)\\
&\leq P_0\left(\left\|\tilde{P}(f_0)\right\|_2\leq2 \left\|\hat{f}_n-\tilde{F}(f_0)\right\|_2\right)+P_0\left(\left\|\tilde{P}(f_0)\right\|_2\leq 2Lr_{N,\gamma}\right)
\end{align*}
and we show below that both probabilities on the right hand side tend to zero.

The first term disappears in view of assumption \eqref{cond:biasLB} and \eqref{hulp:new:star}. For the second term note, that in view of Markov's inequality and $P_0(\|W_N\|_2^2\geq r_{N,\gamma}^2|\mathbb{X})=\gamma$, where $W_N$ is a centered GP with covariance kernel $\hat{C}_N$, we have $\gamma r_{N,\gamma}^2\leq E[\|W_N\|_2^2|\mathbb{X}]$. Then
\begin{align}
P_0(2Lr_{N,\gamma}^2\geq \|\tilde{P}(f_0)\|^2_2)&\leq P_0\Big( E[\|W_N\|_2^2|\mathbb{X}]\geq \frac{\gamma}{2L}\|\tilde{P}(f_0)\|^2_2\Big)\nonumber\\
&\leq\frac{2LE_0(\int_{\mathcal{X}}\Var(f(x)|\mathbb{D}_N)dx)}{\gamma  \|\tilde{P}(f_0)\|^2_2}.\label{eq: hulp:var2}
\end{align}
The expectation in the numerator, known as the \textit{learning curve}, is of order $(\sigma^2/N)\sum_{j\in\mathbb{N}^d}\nu_j$ according to Lemma \ref{lem: bound exp post var}; thus for all $L>0$ not depending on $N$ the right hand side of the preceding display goes to $0$  in view of assumption \eqref{cond:biasLB}. 

\subsubsection{Distributed setting}
\textbf{Preliminary results.} We start by introducing the distributed version of the notations introduced in Section \ref{sec:nondistr:cov}. The aggregated posterior covariance function is $\hat{C}^I_{n,m}(x,x')=m^{-2}\sum_{k=1}^m\hat{C}_{n}^{I,(k)}(x,x')$, where the local posterior covariance functions can be given as 
$\hat{C}_{n}^{I,(k)}(x,x')=K_x^I(x')-\hat{K}_{x,n}^{I,(k)}(x')$ with 
\begin{align*}
\hat{K}_{x,n}^{I,(k)}(\cdot)&=K^I(\cdot,\mathbb{X}^{(k)}_n)[K^I(\mathbb{X}^{(k)}_n,\mathbb{X}^{(k)}_n)+\sigma^2I_n]^{-1}K^I(\mathbb{X}^{(k)}_n,x)\\
&=mK(\cdot,\mathbb{X}^{(k)}_n)[K(\mathbb{X}^{(k)}_n,\mathbb{X}^{(k)}_n)+m^{-1}\sigma^2I_n]^{-1}K(\mathbb{X}^{(k)}_n,x).
\end{align*}
Then in view of \eqref{eq: cov:KRR},
 \begin{align*}
m^{-1}\hat{K}_{x,n}^{I,(k)}&
=\arg\min_{g\in\mathcal{H}}\frac{1}{n}\Big[\sum_{i=1}^n(K_x(X_i^{(k)})-g(X_i^{(k)}))^2+\frac{\sigma^2}{m}\|g\|_{\mathcal{H}}^2\Big].
\end{align*}
For convenience let us introduce the notation $\tilde{K}_{x,n}^{(k)}=m^{-1}\hat{K}_{x,n}^{I,(k)}$.
Then the corresponding score function (up to constant multipliers) is given by
\begin{align*}
\hat{S}_{K_x,n}^{I,(k)}(g)=n^{-1}\Big(\sum_{i=1}^n \big(K_x(X_i^{(k)})-g(X_i^{(k)})\big)K_{X_i^{(k)}}-\frac{\sigma^2}{m} g \Big)
\end{align*}
satisfying $\hat{S}_{K_x,n}^{I,(k)}( \tilde{K}_{x,n}^{(k)})=0$. Furthermore the expected value of the score function is
\begin{align*}
{S}_{K_x,n}^{I}(g)=E\hat{S}_{K_x,n}^{I,(k)}(g)= \int_{\mathcal{X}}\big(K_x(z)-g(z) \big)K_z dz-\frac{\sigma^2}{N}g={S}_{K_x,N}(g),
\end{align*}
hence ${S}_{K_x,n}^{I}(\tilde{F}(K_x))=0$.

Then similarly to the posterior mean in Section \ref{sec:nondistributed:KRR} the following assertions hold
\begin{align}
\Delta \tilde{K}_{x,n}^{(k)}&=  \tilde{K}_{x,n}^{(k)}-\tilde{F}(K_x)=-\tilde{F}\circ F^{-1} \circ {S}_{K_x,n}^{I}( \tilde{K}_{x,n}^{(k)}),\nonumber\\
\hat{S}_{K_x,n}^{I,(k)}(\tilde{F}(K_x))&=\frac{1}{n}\Big( \sum_{i=1}^n \tilde{P}(K_x)(X_i^{(k)})K_{X_i^{(k)}}-\frac{\sigma^2}{m} \tilde{F}(K_x) \Big),\label{eq:cov:help3:distr}\\
F\circ \tilde{F}^{-1}&(\Delta \tilde{K}_{x,n}^{(k)})-\hat{S}_{K_x,n}^{I,(k)}( \tilde{F}(K_x))\nonumber \\
&= - \frac{1}{n} \sum_{i=1}^n \Delta \tilde{K}_{x,n}^{(k)}(X_i^{(k)})K_{X_i^{(k)}}+\int_{\mathcal{X}} \Delta \tilde{K}_{x,n}^{(k)}(x')K_{x'} dx'. \label{eq:cov:help4:distr}
\end{align}

\textbf{Main assertions.} Similarly to the nondistributed case in Section \ref{sec:nondistr:cov}, for the coverage of the credible sets it is sufficient to show that
\begin{align}
\label{eq: bound radius2 - dist}
{P_0\Big(}r_{n,m}^2(\gamma)\geq C_2\frac{\sigma^2}{N}\sum_{j\in\mathbb{N}^d}\nu_j{\Big)\to 1},\\
\label{eq: bound var - dist}
P_0\big(\|\hat{f}_{n,m}-\tilde{F}(f_0)\|_2^2\leq L_N\frac{\sigma^2}{N}\sum_{j\in\mathbb{N}^d}\nu_j\big)\to 1,
\end{align}
where the radius $r_{n,m}(\gamma)$ is defined as $P(\|W_{n,m}\|_2^2\leq r_{n,m}^2(\gamma)|\mathbb{X})=1-\gamma$ and $W_{n,m}$ is a centered GP with the same covariance kernel as $\Pi^{\dagger}(.|\mathbb{D}_N)$. Furthermore, the lack of coverage under \eqref{cond:biasLB} follows from
\begin{align}
\label{eq: bound radius1 - dist}
P_0\big(Lr_{n,m}^2(\gamma)\geq \|\tilde{P}(f_0)\|^2_2\big)\rightarrow 0.
\end{align}
We prove below the above assertions.\\

\textbf{Proof of \eqref{eq: bound radius2 - dist}:} 
Similarly to the proof of \eqref{eq: bound radius - undersmth} we get by Chebyshev's inequality that
$$r_{n,m}^2(\gamma)\geq E[\|W_{n,m}\|_2^2|\mathbb{X}]-(1-\gamma)^{-1/2}\Var(\|W_{n,m}\|_2^2|\mathbb{X})^{1/2}.$$
Then in view of 
\begin{align}
\Var^I_{n,m}(f(x))=m^{-2}\sum_{k=1}^m\Var^I(f(x)|\mathbb{D}_{n}^{(k)}),\qquad \text{for all $x\in\mathcal{X}$},\label{eq:var:distr}
\end{align}
 and Lemma \ref{lem: bound exp post var} it holds almost surely that
\begin{align}
E[\|W_{n,m}\|_2^2|\mathbb{X}]=\int_{x\in\mathcal{X}}\Var^I_{n,m}(f(x)) dx \gtrsim\frac{\sigma^2}{N}\sum_{j\in\mathbb{N}^d}\nu_j.\label{eq: hulp2}
\end{align}

Furthermore, as in \eqref{eq: help2},
$$\Var(\|W_{n,m}\|_2^2|\mathbb{X})=2\int_{\mathcal{X}}\|\hat{C}^I_{n,m}(x,.)\|_2^2dx.$$
Recall that the covariance function $\hat{C}^I_{n,m}(x,x')=m^{-2}\sum_{k=1}^m\hat{C}_{n}^{I,(k)}(x,x')$. Then in view of $(\sum_{i=1}^m a_i)^2\leq m\sum_{i=1}^m a_i^2$,
\begin{align*}
\|\hat{C}^I_{n,m}(x,.)\|_2^2=\Big\|m^{-2}\sum_{k=1}^m\hat{C}_{n}^{I,(k)}(x,.)\Big\|_2^2
\leq m^{-3}\sum_{k=1}^m\|\hat{C}_{n}^{I,(k)}(x,.)\|_2^2.
\end{align*}
We show below that
\begin{align}
E_0\|\hat{C}_{n}^{I,(k)}(x,.)\|_2^2&\lesssim m^2\big(\|\tilde{P}(K_x)\|_2^2+\tilde\delta_{N}\big).\label{eq: bound post_cov:dist}
\end{align}
for $\tilde\delta_{N}= \inf\{ (\sum_{j\in\mathbb{N}^d}\nu_j^2)^2 \sum_{\ell\in \mathcal{I}^c}\mu_\ell:\, |\mathcal{I}|\leq n/(m\sum_{j\in\mathbb{N}^d}\nu_j^2)\}$ similarly to the non-distributed case. Then in view of assertion \eqref{eq:tilde:P:K}, the variance of $\|W_{n,m}\|_2^2$, similarly to \eqref{eq: help3}, is bounded from above by
\begin{align*}
E_0\Var(\|W_{n,m}\|_2^2|\mathbb{X})&=2\int_{\mathcal{X}}E_0\|\hat{C}^I_{n,m}(x,.)\|_2^2dx\\
&\lesssim \Big(\int_{\mathcal{X}}\|\tilde{P}(K^I_x)\|_2^2dx+\tilde\delta_{N}\Big)\\
&=\frac{\sigma^4}{N^2}\sum_{j\in\mathbb{N}^d}\nu_j^2+\tilde\delta_{N}.
\end{align*}
Hence for all $t>0$ we get by Markov's inequality and Lemmas \ref{lem: bound eigen poly} and \ref{lem: bound eigen expo} that
\begin{align*}
P_{0}\Big(\Var(\|W_{n,m}\|_2^2|\mathbb{X})& \geq t \frac{\sigma^4}{N^2}\big(\sum_{j\in\mathbb{N}^d}\nu_j\big)^2\Big)\\
&\lesssim t^{-2}\Big(\frac{\sum_{j\in\mathbb{N}^d}\nu_j^2}{ \big(\sum_{j\in\mathbb{N}^d}\nu_j\big)^2}+\frac{\tilde\delta_{N} N^2}{ \sigma^4\big(\sum_{j\in\mathbb{N}^d}\nu_j\big)^2}\Big)=o(1).
\end{align*}
Hence with $P_{0}$-probability tending to one $E[\|W_{n,m}\|_2^2|\mathbb{X}]$ is of higher order than $\Var (\|W_{n,m}\|_2^2)^{1/2}$. Therefore, the quantiles of $\|W_{n,m}\|_2^2$ are of the order $(\sigma^2/N)\sum_{j\in\mathbb{N}^d}\nu_j$ {with $P_0$-probability going to one}, including $r_{n,m}^2(\gamma)$.\\

\textbf{Proof of \eqref{eq: bound post_cov:dist}:} We adapt the proof of \eqref{eq: bound post_cov} to the distributed setting. First note that
\begin{align*}
\|\hat{C}_{n}^{I,(k)}(x,.)\|_2^2&\lesssim m^2\Big( \|\tilde{P}(K_x)\|_2^2+ \| \Delta \tilde{K}_{x,n}^{(k)}- \tilde{F}\circ F^{-1} \circ\hat{S}_{K_x,n}^{I,(k)}(\tilde{F}(K_{x}))  \|_2^2\\
&\qquad\quad +\|  \tilde{F}\circ F^{-1}\circ \hat{S}_{K_x,n}^{I,(k)}(\tilde{F}(K_{x}))  \|_2^2\Big),
\end{align*}
where $\Delta \tilde{K}_{x,n}^{(k)}=\hat{K}_{n,x}^{I,(k)}/m-\tilde{F}(K_x)$. Then, in view of \eqref{eq:cov:help3:distr}, we get that
\begin{align*}
E_0\|  \tilde{F}&\circ F^{-1}\circ \hat{S}_{K_x,n}^{I,(k)}(\tilde{F}(K_{x}))  \|_2^2\\ 
&=E_0\Big\|  \tilde{F}\circ F^{-1}\Big(\hat{S}_{K_x,n}^{I,(k)}( \tilde{F}\big( K_{x}))- {S}_{K_x,n}^{I,(k)}( \tilde{F}\big( K_{x})) \Big) \Big \|_2^2\\ 
&= E_0\Big\|\tilde{F} \circ F^{-1}\Big(\frac{1}{n} \sum_{i=1}^{n}\tilde{P}(K_x)(X_i^{(k)})K_{X_i^{(k)}}-E_X[\tilde{P}(K_x)(X)K_{X}]\Big)\Big\|_2^2\\
&\leq
\Big(\frac{1}{n}\sum_{j\in\mathbb{N}^d}\nu_j^2\Big)\|\tilde{P}(K_x)\|_2^2=o( \|\tilde{P}(K_x)\|_2^2),
\end{align*}
where the penultimate inequality follows from Lemma \ref{lem: bound} with $g=\tilde{P}(K_x)$.

Furthermore, similarly to the proof in Section \ref{sec:proof: hulp:var1}, by using assertion \eqref{eq:cov:help4:distr}, Lemma \ref{lem:bound:random} (with $\hat{g}^{(k)}=\Delta \tilde{K}_{x,n}^{(k)}$ and sample size $n$) and Lemma \ref{lem:cov:UB:help}, we can show that for all $x\in\mathcal{X}$
\begin{align*}
&E_0\Big\|\Delta \tilde{K}_{x,n}^{(k)}- \tilde{F} \circ F^{-1}\circ \hat{S}_{K_x,n}^{I,(k)}(\tilde{F}(K_{x}))\Big\|_2^2\nonumber\\
&\qquad= E_0\Big\|(\tilde{F}\circ F^{-1})\Big(\frac{1}{n}\sum_{i=1}^{n}\Delta \tilde{K}_{x,n}^{(k)}(X_i^{(k)})K_{X_i^{(k)}}-\int_{\mathcal{X}}\Delta \tilde{K}_{x,n}^{(k)}(x')K_{x'}dx'\Big)\Big\|_2^2\nonumber\\
&\qquad\lesssim \frac{|\mathcal{I}|\sum_{j\in\mathbb{N}^d}\nu_j^2}{n}E_0\|\Delta \tilde{K}_{x,n}^{(k)}\|_2^2+\tilde\delta_{N}.
\end{align*}
Taking the infimum over $|\mathcal{I}|\leq o\big(n/\sum_{j\in\mathbb{N}^d}\nu_j^2\big)$ we get that the left hand side of hte preceding display is bounded from above by $o(E_0\|\Delta \tilde{K}_{x,n}^{(k)}\|_2^2)+\tilde\delta_{N}$. We conclude the proof of \eqref{eq: bound post_cov:dist} by combining the above three displays.\\

\textbf{Proof of \eqref{eq: bound var - dist}:}  Exactly the same as the proof of \eqref{eq: bound var - undersmth}.\\

\textbf{Proof of \eqref{eq: bound radius1 - dist}:} Similarly to assertion \eqref{eq: hulp:var2} we get in view of \eqref{eq:var:distr} and Lemma \ref{lem: bound exp post var} in the case where assumption \eqref{cond:biasLB} holds
\begin{align*}
P_0\big(L r_{n,m}^2(\gamma)\geq \|\tilde{P}(f_0)\|^2_2\big)&\leq 
\frac{2L {E_0}\int_{\mathcal{X}}\Var_{n,m}(f(x))dx.}{\gamma  \|\tilde{P}(f_0)\|^2_2}\\
&\lesssim \frac{m^{-1}\sigma^2 \sum_{j\in \mathbb{N}^d}\frac{m\mu_j}{\sigma^2+nm\mu_j}}{\|\tilde{P}(f_0)\|^2_2}\\
&= \frac{N^{-1}\sigma^2 \sum_{j\in \mathbb{N}^d}\nu_j}{\|\tilde{P}(f_0)\|^2_2}=o(1).
\end{align*}



 


\section{Proof of the Corollaries}\label{sec:corollaries}

\subsection{Proof of Corollary \ref{cor: contraction:poly}}\label{sec: proof_contraction:poly}
First note that for any $\mathcal{N}\subset\mathbb{N}^d$
\begin{align}
\|\tilde{P}(f_0)\|_2^2&=\sum_{j\in\mathbb{N}^d}(1-\nu_j)^2f_{0,j}^2=\sum_{j\in\mathbb{N}^d}\frac{\sigma^4}{(\sigma^2+\mu_j N)^2}f_{0,j}^2\nonumber\\
&\leq(N/\sigma^2)^{-2}\sum_{j\in\mathcal{N}}\frac{1}{\mu_j^2}f_{0,j}^2+\sum_{j\in\mathbb{N}^d/\mathcal{N}}f_{0,j}^2.\label{eq: cov:poly:pos:help}
\end{align}

Consider  eigenvalues satisfying \eqref{assump:poly} with $\alpha=\beta$, i.e. $\mu_j\asymp\big(\prod_{i=1}^d j_i\big)^{-2\beta/d-1}$. Let us take $\mathcal{N}=\{j\in\mathbb{N}^d:\Pi_{i=1}^d j_i\leq J_{\beta}\}$ with $J_{\beta}:=(N/\sigma^2)^{\frac{d}{2\beta+d}}$ and note that  in view of \eqref{def:small_index} [with $I=J_{\beta}$] we have
\begin{align}
|\mathcal{N}|\lesssim J_{\beta}\log^{d-1} J_\beta=o(N)\label{eq:UB:size}
\end{align}
Furthermore, we also get that
\begin{align*}
\sup_{f_0\in \Theta^{\beta}(B)}\|\tilde{P}(f_0)\|_2^2
&\lesssim \sup_{f_0\in \Theta^{\beta}(B)}\Bigg[\frac{\sigma^4}{N^2}\max_{j\in\mathcal{N}}\left(\prod_{i=1}^d j_i\right)^{4\beta/d+2}\left(\sum_{i=1}^d j_i^2\right)^{-\beta}\sum_{j\in\mathcal{N}} \left(\sum_{i=1}^d j_i^2\right)^{\beta}f_{0,j}^2\\
&\qquad\qquad\qquad+\sup_{j\notin\mathcal{N}}\left(\sum_{i=1}^d j_i^2\right)^{-\beta}  \sum_{j\notin\mathcal{N}} \left(\sum_{i=1}^d j_i^2\right)^{\beta}f_{0,j}^2\Bigg]\\
&\lesssim\left(\frac{N}{\sigma^2}\right)^{-2}J_{\beta}^{2\beta/d+2}B^2+J_{\beta}^{-2\beta/d}B^2\\
&\lesssim(N/\sigma^2)^{\frac{-2\beta}{2\beta+d}},
\end{align*}
using Lemmas \ref{lem: bound-maxfunc} [with $r=4\beta/d+2$, $s=2\beta$ and $J=J_{\beta}$] and \ref{lem: bound-minfunc} [with $s=2\beta$ and $J=J_{\beta}$].

Moreover, in view of Lemma \ref{lem: bound eigen poly} and $\nu_j\leq 1$
$$\frac{\sigma^2}{N}\sum_{j\in\mathbb{N}^d}\nu_j^2\asymp  \frac{\sigma^2}{N}J_{\beta}\left(\log J_{\beta}\right)^{d-1}=(N/\sigma^2)^{\frac{-2\beta}{2\beta+d}}\left(\log(N/\sigma^2)\right)^{d-1}.$$

Finally we show that the remaining term is $\delta_{N}=o(N^{\frac{-2\beta}{d+2\beta}})$. Let us take
$$\mathcal{I}=\left\{ j\in\mathbb{N}^d:\, \prod_{i=1}^dj_i\leq I \right\}\qquad \text{with }I=\frac{N}{m^2\log^{d-1}(n)}\Big(\sum_{j\in\mathbb{N}^d}\nu_j^2\Big)^{-1},$$
where $\frac{N}{m^2\log^{d-1}(n/m)}\Big(\sum_{j\in\mathbb{N}^d}\nu_j^2\Big)^{-1}\geq 1$ holds because $m$ is small enough. Note that in view of Lemma \ref{lem:card} the cardinality of $\mathcal{I}$ satisfies $|\mathcal{I}|\lesssim \frac{n}{m}\Big(\sum_{j\in\mathbb{N}^d}\nu_j^2\Big)^{-1}$, hence it satisfies the cardinality assumption on $\mathcal{I}$. Then in view of Lemma \ref{lem:card} and Lemma \ref{lem: bound eigen poly}
\begin{align*}
\delta_{N}&\lesssim N\sum_{j\in\mathbb{N}^d}\nu_j^2 \sum_{ \ell:\, \prod_{i=1}^d \ell_i>I}\mu_\ell\lesssim N\sum_{j\in\mathbb{N}^d}\nu_j^2 I^{-2\beta/d}\log^{d-1}I\\
&\ll N^{1-2\beta/d}m^{4\beta/d}\big(\sum_{j\in\mathbb{N}^d}\nu_j^2\big)^{2\beta/d+1}(\log N)^{2\beta+d-1}\\
&\lesssim N^{2-2\beta/d}m^{4\beta/d}(\log N)^{2\beta+d-1}.
\end{align*}
The right hand side is of order $o(N^{-\frac{2\beta}{d+2\beta}})$ for all $m=o(N^{\frac{2\beta-3d}{4\beta}})$ with $\beta> 3d/2$. Combining the above inequality with Theorem \ref{thm: contr:general} concludes the proof for the polynomially decaying eigenvalues.

\subsection{Proof of Corollary \ref{cor: contraction:exp}}\label{sec: proof_contraction:exp}
For arbitrary index set $\mathcal{N}\subset \mathbb{N}^d$ we get that
\begin{align}
\sup_{f_0\in\Theta^{\beta}(B)} \|\tilde{P}(f_0)\|_2^2
&\leq \sup_{f_0\in\Theta^{\beta}(B)} \Bigg[\frac{\sigma^4}{N^2}\max_{j\in\mathcal{N}}\left(\sum_{i=1}^d j_i^2\right)^{-\beta}e^{2a\sum_{i=1}^dj_i} \sum_{j\in\mathcal{N}} \left(\sum_{i=1}^d j_i^2\right)^{\beta}f_{0,j}^2\nonumber\\
&\qquad\qquad\qquad+\sup_{j\notin\mathcal{N}}\left(\sum_{i=1}^d j_i^2\right)^{-\beta}  \sum_{j\notin\mathcal{N}} \left(\sum_{i=1}^d j_i^2\right)^{\beta}f_{0,j}^2\Big].\label{eq: help1}
\end{align}
We deal with the two terms on the right hand side separately. Note that the function $x\mapsto x^{-2\beta}e^{2ax}$ is convex on $[1,J_a]$, for $J_a=a^{-1}\log (N/\sigma^2)$ with $a\leq 1$, and achieves its maximum at one of the end points. Let us take the set $\mathcal{N}=\{j\in\mathbb{N}^d:\,\sum_{i=1}^d j_i\leq J_{a}\}$ and note that 
\begin{align}
|\mathcal{N}|\leq a^{-d}\log^d N=o(N),\label{eq:cond:exp:size}
\end{align}
by the lower bound on $a$. Furthermore, by noting that $(\sum_{i=1}^d j_i)^2\leq d\sum_{i=1}^d j_i^2$, the 
maximum of the last display over $\mathcal{N}$ is bounded from above by
$$\max_{j\in\mathcal{N}}\left(\sum_{i=1}^d j_i^2\right)^{-\beta}e^{2a\sum_{i=1}^dj_i}
\lesssim \max_{j\in\mathcal{N}}\left(\sum_{i=1}^d j_i\right)^{-2\beta}e^{2a\sum_{i=1}^dj_i}
\lesssim 1+J_a^{-2\beta}e^{2aJ_a}.$$

The second term in \eqref{eq: help1} is directly bounded from above by $J_a^{-2\beta}B^2$.
Therefore, by combining the inequalities above, 

\begin{align}
\|\tilde{P}(f_0)\|_2^2\lesssim\frac{\sigma^4}{N^2}+\left(a^{-1}\log (N/\sigma^2)\right)^{-2\beta}.\label{eq: UBbias}
\end{align}
Moreover, in view of Lemma \ref{lem: bound eigen expo}
\begin{align}
\frac{\sigma^2}{N}\sum_{j\in\mathbb{N}^d}\nu_j\asymp\frac{\sigma^2}{N}J_a^d=\frac{\sigma^2}{N}a^{-d}\log^{d}(N/\sigma^2).\label{eq: help:UBbias}
\end{align}
For $a:=(N/\sigma^2)^{-\frac{1}{2\beta+d}}\log(N/\sigma^2)$ both of the preceding displays are bounded from above by a multiple of $(N/\sigma^2)^{-\frac{2\beta}{2\beta+d}}$.

Finally, we show that the remainder term $\delta_{N}$ is of lower order than $(N/\sigma^2)^{-\frac{2\beta}{2\beta+d}}$. We take $\mathcal{I}=\{j\in\mathbb{N}^d:\,\sum_{i=1}^d j_i\leq I\}$, with $I=N^{1/d}\big(m^2\sum_{j\in\mathbb{N}^d}\nu_j^2\big)^{-1/d}$. Then it is easy to see that $|\mathcal{I}|\leq I^d\leq N\big(m^2\sum_{j\in\mathbb{N}^d}\nu_j^2\big)^{-1}$ holds. Note that $|\mathcal{I}|\geq 1$ holds because $m$ is small enough. Furthermore, in view of the upper bound $p(i,d)\leq \frac{1}{2}{i-1 \choose d-1}+1/2\leq i^d$ on the $d$ partition of $i\in\mathbb{N}$
, we get that
\begin{align}
\delta_{N}&\leq n\sum_{j\in\mathbb{N}^d}\nu_j^2\sum_{\ell\in\mathcal{I}^c}\mu_{\ell}
\leq n\sum_{j\in\mathbb{N}^d}\nu_j^2 \sum_{i\geq I} i^d e^{-ai}\nonumber\\
&\lesssim nI^de^{-aI}\sum_{j\in\mathbb{N}^d}\nu_j^2\lesssim \left(\frac{n}{m}\right)^{2}e^{-aI}(\log n)^{-1}.\label{eq: remainder_exp}
\end{align}
Since $\beta\geq d/2$, we have
\begin{align*}
a I=\left(\frac{N}{\sigma^2}\right)^{-\frac{1}{2\beta+d}}\log\left(\frac{N}{\sigma^2}\right)N^{\frac{1}{d}}m^{-2/d}\big(\sum_{j\in\mathbb{N}^d}\nu_j^2\big)^{-1/d}\\
\gtrsim N^{\frac{2\beta-d}{d(2\beta+d)}}m^{-2/d}\log N\geq L\log N.
\end{align*}
Hence the right hand side of \eqref{eq: remainder_exp} is $o(N^{-L})$, for arbitrary $L>0$, when $m=o(N^{\frac{\beta-d/2}{2\beta+d}})$ concluding the proof of the corollary using Theorem \ref{thm: contr:general}.

 \subsection{Proof of Corollary \ref{cor: coverage:poly}}\label{sec:cor: coverage:poly}
We proceed by proving that the conditions of Theorem \ref{th: unc quant} hold for this choice of the kernel and the parameters, which directly provides us the statements.

Let us take $\mathcal{N}=\{j\in\mathbb{N}^d:\prod_{i=1}^d j_i\leq J_{\alpha}\}$ with $J_{\alpha}:=(N/\sigma^2)^{1/(d+2\alpha)}$ in $\eqref{eq: cov:poly:pos:help}$. The cardinality of this set is $o(N)$, see \eqref{eq:UB:size}. Furthermore, in view of $\alpha\leq\beta$,
\begin{align*}
\sup_{f_0\in \Theta^{\beta}(B)}\|\tilde{P}(f_0)\|_2^2
&\lesssim \sup_{f_0\in \Theta^{\beta}(B)}\Bigg[\frac{\sigma^4}{N^2}\max_{j\in\mathcal{N}}\big(\prod_{i=1}^d j_i\big)^{4\alpha/d+2}\big(\sum_{i=1}^dj_i^2\big)^{-\beta}\sum_{j\in\mathcal{N}} \big(\sum_{i=1}^dj_i^2\big)^{\beta}f_{0,j}^2\\
&\qquad\qquad\qquad+\sup_{j\notin\mathcal{N}}\big(\sum_{i=1}^dj_i^2\big)^{-\beta}  \sum_{j\notin\mathcal{N}} \big(\sum_{i=1}^dj_i^2\big)^{\beta}f_{0,j}^2\Bigg]\\
&\lesssim\left(\frac{N}{\sigma^2}\right)^{-2}J_{\alpha}^{4\alpha/d-2\beta/d+2}B^2+J_{\alpha}^{-2\beta/d}B^2\lesssim(N/\sigma^2)^{\frac{-2\beta}{2\alpha+d}}.
\end{align*}

Then, in view of Lemma \ref{lem: bound eigen poly}, $\nu_j\leq 1$ and the preceding display,
\begin{align*}
\frac{\sigma^2}{N}\sum_{j\in\mathbb{N}^d}\nu_j&\asymp  \frac{\sigma^2}{N}J_{\alpha}\left(\log J_{\alpha}\right)^{d-1}
=\left(\frac{N}{\sigma^2}\right)^{-\frac{2\alpha}{2\alpha+d}}\left(\log\left(\frac{N}{\sigma^2}\right)\right)^{d-1}\gtrsim \sup_{f_0\in \Theta^{\beta}(B)}\|\tilde{P}(f_0)\|_2^2,
\end{align*}
when $\alpha\leq\beta$. Finally in view of Corollary \ref{cor: contraction:poly} we have that 
$$\delta_{N}=o\left((N/\sigma^2)^{-\frac{2\alpha}{2\alpha+d}}\right)=o\Big( \frac{\sigma^2}{N}\sum_{j\in\mathbb{N}^d}\nu_j\Big),$$
finishing the proof of the corollary.

 \subsection{Proof of Corollary \ref{cor: coverage:exp}}\label{sec:cor: coverage:exp}
We again prove that the conditions of Theorem \ref{th: unc quant} hold in this setting.

In view of assertions \eqref{eq: UBbias}  and \eqref{eq: help:UBbias}, we get for $a\lesssim\big(\frac{\sigma}{N}\big)^{1/(2\beta+d)}\log\big(\frac{N}{\sigma}\big)$ that
\begin{align*}
\|\tilde{P}(f_0)\|_2^2\lesssim \frac{\sigma^2}{N}\sum_{j\in\mathbb{N}^d}\nu_j.
\end{align*}

 Furthermore, the cardinality of the set $\{j\in \mathbb{N}^{d}:\,  \mu_j N \geq \sigma^2\}$ is $o(N)$, see \eqref{eq:cond:exp:size}. Finally, in view of Corollary \ref{cor: contraction:exp}, $\delta_{N}=o(N^{-c})$, hence the condition $\delta_{N}=o\Big( \frac{\sigma^2}{N}\sum_{j\in\mathbb{N}^d}\nu_j\Big)$ of Theorem \ref{th: unc quant} also holds, concluding the proof.

\appendix
\section{Technical lemmas}\label{sec: lemmas}

\begin{lemma}\label{lem: bound}
Consider the local regression problem \eqref{eq: nonpara reg dist} for arbitrary $k\in\{1,...,m\}$ and let $g\in L_2(\mathcal{X})$. Then there exists a universal constant $C$ not depending on $g$ such that
\begin{align}
\label{eq: bound lemma}
E_0\Big\|(\tilde{F}\circ F^{-1})\Big(\frac{1}{n}\sum_{i=1}^ng(X_i^{(k)})K_{X_i^{(k)}}-E_X[g(X)K_X]\Big)\Big\|_2^2\leq\frac{C}{n}\|g\|_2^2\sum_{j\in\mathbb{N}^d}\nu_j^2,
\end{align}
where $X$ is a uniform random variable on $\mathcal{X}$, and $\nu_j$'s are the eigenvalues of the operator $\tilde{F}$.
\end{lemma}

\begin{proof}
For simplicity we omit the reference to the local $k$ machine in the proof by writing $X_i=X_i^{(k)}$. Let $g=\sum_{j\in\mathbb{N}^d}g_j\psi_j\in L_2(\mathcal{X})$. Since $g(X)K_X=\sum_{j,k\in\mathbb{N}^d}\mu_jg_k\psi_j(X)\psi_k(X)\psi_j$ and  $(\psi_j)_{j\in\mathbb{N}^d}$ is an orthonormal basis of $L_2(\mathcal{X})$, we have $E_X[g(X)K_X]=\sum_{j\in\mathbb{N}^d}\mu_jg_j\psi_j.$
Furthermore, the linearity of the operator $\tilde{F}\circ F^{-1}$ implies that $\tilde{F}\circ F^{-1}(g(X)K_{X})=\sum_{j,k\in\mathbb{N}^d}\nu_jg_k\psi_j(X)\psi_k(X)\psi_j$, providing
\begin{align}
\tilde{F}\circ F^{-1}\big(E_X[g(X)K_X]\big)&=\sum_{j\in\mathbb{N}^d}\nu_jg_j\psi_j,\nonumber\\
\tilde{F}\circ F^{-1}\Big(\frac{1}{n}\sum_{i=1}^ng(X_i)K_{X_i}\Big)&=\frac{1}{n}\sum_{i=1}^n\tilde{F}\circ F^{-1}(g(X_i)K_{X_i})\nonumber\\
&=\frac{1}{n}\sum_{i=1}^n\sum_{j,k\in\mathbb{N}^d}\nu_jg_k\psi_j(X_i)\psi_k(X_i)\psi_j.\label{eq: help:tildf}
\end{align}
Then using the inequality $(a+b)^2\leq2(a^2+b^2)$ we get
\begin{align*}
&E_0\Big\|(\tilde{F}\circ F^{-1})\Big(\frac{1}{n}\sum_{i=1}^ng(X_i) K_{X_i}-E_X[g(X) K_X]\Big)\Big\|_2^2\\
&\qquad=E_0\Big\|\sum_{j,k\in\mathbb{N}^d}\nu_jg_k\psi_j \big(\frac{1}{n}\sum_{i=1}^n\psi_j(X_i)\psi_k(X_i)-\delta_{jk}\big)\Big\|_2^2\\
&\qquad=\sum_{j\in\mathbb{N}^d}\frac{\nu_j^2}{n} E_0\big(g(X_i)\psi_j(X_i)-g_j\big)^2\\
&\qquad\leq 2\sum_{j\in\mathbb{N}^d}\frac{\nu_j^2}{n} \Big(  E_0g^2(X_i)\psi_j^2(X_i) +g_j^2 \Big)\leq \frac{2(C_{\psi}^2+1)  \|g\|_2^2}{n} \sum_{j\in\mathbb{N}^d}\nu_j^2, 
\end{align*}
finishing the proof of the statement.

\end{proof}

\begin{lemma}\label{lem:bound:random}
Consider the local regression problem \eqref{eq: nonpara reg dist} for arbitrary $k\in\{1,...,m\}$. Then for any finite index set $\mathcal{I}\subset\mathbb{N}^d$, $|\mathcal{I}|\leq N^C$ and data dependent function $\hat{g}^{(k)}:\mathcal{X}^{n}\mapsto \mathbb{R}$,  $\|\hat{g}^{(k)}\|_2\leq N^C$, for some $C>0$,
\begin{align}
\label{eq: bound random lemma}
E_0&\Big\|(\tilde{F}\circ F^{-1})\Big(\frac{1}{n}\sum_{i=1}^{n} \hat{g}^{(k)}(X_i^{(k)})K_{X_i^{(k)}}-E_X[\hat{g}^{(k)}(X) K_X]\big)\Big)\Big\|_2^2\nonumber\\ 
&\lesssim\quad\frac{|\mathcal{I}| \log N}{n}\sum_{j\in\mathbb{N}^d}\nu_{j}^2E_0\|\hat{g}^{(k)}\|_2^2+ E_0 \| \hat{g}_{\mathcal{I}^c}^{(k)}\|_{\mathcal{H}}^2\sum_{j\in\mathbb{N}^d}\nu_{j}^2  \sum_{\ell  \in\mathcal{I}^c}\mu_\ell+N^{-C_0},
\end{align}
where $X$ is a uniform random variable on $\mathcal{X}$, $\nu_{j}$'s are the eigenvalues of the operator $\tilde{F}$, $C_0$ can be chosen arbitrarily large, and  $\hat{g}_{\mathcal{I}^c}^{(k)}(\cdot)=\sum_{j\in \mathcal{I}^c} \hat{g}_j^{(k)} \psi_j(\cdot)$.
\end{lemma}

\begin{proof}
For simplicity we omit the reference to the $k$th local problem and write $X_i=X_i^{(k)}$ and $\hat{g}=\hat{g}^{(k)}$.
Let us next define the event
\begin{align}
\mathcal{A}_{\mathcal{I},j}&=\Big\{ (X_1,...,X_n)\in\mathcal{X}^n:\,\Big(\frac{1}{n}\sum_{i=1}^{n} \psi_{j}(X_i)\psi_\ell(X_i)-\delta_{j\ell} \Big)^2\leq \frac{8C_\psi^2 C \log N}{n}, \forall\ell\in \mathcal{I}\Big\}.
\end{align}
Note that by Hoeffding's inequality, for arbitrary $\ell\in\mathcal{I}$,
\begin{align*}
P(\mathcal{A}_{\mathcal{I},j}^c)&\leq  |\mathcal{I}| P\Big( \big(\frac{1}{n}\sum_{i=1}^{n} \psi_{j}(X_i)\psi_\ell(X_i)-\delta_{j\ell} \big)^2> \frac{8C_\psi^2C \log N}{n}\Big)\\
&\leq 2|\mathcal{I}| \exp\Big\{-\frac{4C_\psi^2C \log N }{C_\psi^2}\Big\} \leq O(|\mathcal{I}| N^{-3C}).
\end{align*}
Then using $(a+b)^2\leq 2a^2+2b^2$ and Cauchy-Schwarz inequality
\begin{align*}
&E_0\Big\|(\tilde{F}\circ F^{-1})\Big(\frac{1}{n}\sum_{i=1}^{n}\hat{g}(X_i) K_{X_i}-E_X[\hat{g}(X) K_X]\Big)\Big\|_2^2\\
&\qquad=E_0\Big\|\sum_{j\in\mathbb{N}^d}\sum_{\ell\in\mathbb{N}^d}\nu_j\hat{g}_\ell\psi_j \big(\frac{1}{n}\sum_{i=1}^{n}\psi_j(X_i)\psi_\ell(X_i)-\delta_{j\ell}\big)\Big\|_2^2\\
&\qquad\lesssim E_0\sum_{j\in\mathbb{N}^d}\nu_j^2 \Big(\sum_{\ell\in \mathcal{I}}\hat{g}_\ell \big(\frac{1}{n}\sum_{i=1}^n\psi_j(X_i)\psi_\ell(X_i)-\delta_{j\ell}\big)\Big)^2\\
&\qquad\qquad + E_0\sum_{j\in\mathbb{N}^d}\nu_j^2 \Big(\sum_{\ell\in \mathcal{I}^c}|\hat{g}_\ell|  (C_{\psi}^2+1) \Big)^2\\
&\qquad\lesssim E_0\sum_{j\in\mathbb{N}^d}\nu_j^2  |\mathcal{I}| \sum_{\ell\in \mathcal{I}}\hat{g}_\ell^2 \big(\frac{1}{n}\sum_{i=1}^n\psi_j(X_i)\psi_\ell(X_i)-\delta_{j\ell}\big)^2\\
&\qquad\qquad + \sum_{j\in\mathbb{N}^d}\nu_j^2 \sum_{\ell \in \mathcal{I}^c}\mu_{\ell}  E_0\sum_{\ell \in \mathcal{I}^c} \hat{g}_\ell^2 \mu_\ell^{-1}\\
&\qquad\leq \sum_{j\in\mathbb{N}^d}\nu_j^2 E_0\|\hat{g}\|_2^2 \Big( \frac{8C_{\psi}^2C|\mathcal{I}|\log N}{n}+1_{A_{j,\mathcal{I}}^c} |\mathcal{I}|\Big)+ E_0\| \hat{g}_{\mathcal{I}^c}\|_{\mathcal{H}}^2  \sum_{j\in\mathbb{N}^d}\nu_j^2 \sum_{\ell \in \mathcal{I}^c}\mu_{\ell}  \\
&\qquad\lesssim \frac{|\mathcal{I}|\log N}{n}\sum_{j\in\mathbb{N}^d}\nu_j^2  E_0\|\hat{g}\|_2^2+ E_0 \| \hat{g}_{\mathcal{I}^c}\|_{\mathcal{H}}^2 \sum_{j\in\mathbb{N}^d}\nu_j^2 \sum_{\ell \in \mathcal{I}^c}\mu_{\ell}+O(N^{-C}),
\end{align*}
where $C$ can be chosen arbitrarily large, concluding the proof of our statement.

\end{proof}

\begin{lemma}\label{lem:cov:UB:help}
There exists $C>0$ such that
\begin{align*}
E_{0}\|\hat{K}_{x,n}^{I,(k)}/m-\tilde{F}(K_x) \|_{\mathcal{H}}^2\leq C \sum_{j\in \mathbb{N}^d}\nu_j^2.
\end{align*}
\end{lemma}

\begin{proof}
First note that
\begin{align*}
\|\hat{K}_{x,n}^{I,(k)}/m-\tilde{F}(K_x) \|_{\mathcal{H}}^2\leq 2m^{-2}\|\hat{K}_{x,n}^{I,(k)} \|_{\mathcal{H}}^2+2\|\tilde{F}(K_x)\|_{\mathcal{H}}^2.
\end{align*}
The second term on the right hand is bounded by
\begin{align*}
\|\tilde{F}(K_x)\|_{\mathcal{H}}^2=\sum_{j\in \mathbb{N}^d}\mu_j^{-1}\nu_j^2 \mu_j^2\psi_j(x)^2
\leq C_{\psi}^2  \sum_{j\in \mathbb{N}^d}\mu_j\nu_j^2 \lesssim \sum_{j\in \mathbb{N}^d}\nu_j^2.
\end{align*}

Since $\hat{K}_{x,n}^{I,(k)}$ is a KRR estimator, we get that
\begin{align*}
E_0\sigma^2 \|\tilde{K}_{x,n}^{(k)}\|_{\mathcal{H}}^2&\leq E_0\Big(\sum_{i=1}^n (\tilde{K}_{x,n}^{(k)}(X_i^{(k)})-K_x(X_i^{(k)}))^2+\sigma^2 \|\tilde{K}_{x,n}^{(k)} \|_{\mathcal{H}}^2\Big)\\
&\leq  E_0\Big( \sum_{i=1}^n (\tilde{F}(K_x)(X_i^{(k)})-K_x(X_i^{(k)}))^2+\sigma^2 \|\tilde{F}(K_x)\|_{\mathcal{H}}^2\Big)\\
&\leq \sum_{i=1}^n E_0\tilde{P}(K_x)^2(X_i^{(k)})+\sigma^2\|\tilde{F}(K_x)\|_{\mathcal{H}}^2=O( \sum_{j\in \mathbb{N}^d}\nu_j^2),
\end{align*}
where the last inequality follows from \eqref{eq:tilde:P:K}.

\end{proof}

\begin{lemma}\label{lem: bound exp post var}
	Assume that the eigenvalues $\mu_j$ of the covariance kernel $K$ satisfy $\sum_{j\in\mathbb{N}^d} \mu_j<\infty$, $|\{j\in \mathbb{N}^{d}:\,  \mu_j N \geq \sigma^2\}|\leq N$, and $\sigma^2\geq c>0$. Then the expectation of the posterior variance is of the following order
	$$E_0E_X\Var\left(f(X)|\mathbb{D}_N\right)\asymp \sigma^2\sum_{j\in\mathbb{N}^d}\frac{\mu_j}{\sigma^2+N\mu_j},$$
	where the expectation $E_X$ corresponds to the random variable $X\sim U[0,1]^d$ and the multiplicative constant depends on $\sum_{j\in\mathbb{N}^d}\mu_j$ and $c$.
\end{lemma}

\begin{proof}
	It is shown in Section 6 of \cite{opper:vivarelli:1999} that the expectation of the posterior variance, named ``generalization error'', is bounded from below as follows
	$$E_0E_X\Var\left(f(X)|\mathbb{D}_n\right)\geq \sigma^2\sum_{j\in\mathbb{N}^d}\frac{\mu_j}{\sigma^2+N\mu_jE_X\psi^2_j(X)}=\sigma^2\sum_{j\in\mathbb{N}^d}\frac{\mu_j}{\sigma^2+N\mu_j}.$$

In \cite{opper:trecate:williams:98}, it has been shown that for stationary GPs, for any $\mathcal{J}\subset \mathbb{N}^d$, with $|\mathcal{J}|\leq N$, the learning curve is bounded from above by
$$E_0E_X\Var\left(f(X)|\mathbb{D}_n\right)\leq\sum_{j\in\mathbb{N}^d}\mu_j-N\sum_{j\in \mathcal{J}}\frac{\mu_j^2}{c_j},$$
where
$$c_j=(N-1)\mu_j+\sigma^2+\sum_{j\in\mathbb{N}^d}\mu_j.$$ 

Let us take $\mathcal{J}=\{j\in \mathbb{N}^{d}:\,  \mu_j N \geq \sigma^2\}$ and by assumption its cardinality is bounded by $N$. Then
\begin{align*}
	\sum_{j\in\mathbb{N}^d}\mu_j-N\sum_{j\in \mathcal{J}}\frac{\mu_j^2}{c_j}&=\sum_{j\in \mathcal{J}}\mu_j\frac{c_j-N\mu_j}{c_j}+\sum_{j\notin \mathcal{J}}\mu_j\\
&	=\sum_{j\in\mathcal{J}}\mu_j\frac{\sum_{j\in\mathbb{N}^d}\mu_j+\sigma^2-\mu_j}{\sum_{j\in\mathbb{N}^d}\mu_j+\sigma^2+(N-1)\mu_j}+\sum_{j\notin \mathcal{J}}\mu_j\\
&\leq \sigma^2\sum_{j\in\mathcal{J}}\mu_j\frac{\sum_{j\in\mathbb{N}^d}\mu_j/\sigma^2+1}{\sigma^2+\mu_j N}+2\sigma^2\sum_{j\notin\mathcal{J}}\frac{\mu_j}{\sigma^2+\mu_j N}\\
&\lesssim \sigma^2\sum_{j\in\mathbb{N}^d} \frac{\mu_j}{\sigma^2+\mu_j N},
\end{align*}
concluding our proof.

\end{proof}

\begin{lemma}\label{lem: bound eigen poly}
For $\nu_j$, $j\in\mathbb{N}^d$, defined in \eqref{eq: eigenvalues equII} with eigenvalues $\mu_j$ polynomially decaying according to Assumption \ref{ass: bound eigenvalues} and $k\in\mathbb{N}$,
\begin{align*}
\sum_{j\in\mathbb{N}^d}\nu_j^k\asymp J_{\alpha}\log^{d-1}J_{\alpha},
\end{align*}
where $J_{\alpha}=(N/\sigma^2)^{\frac{d}{2\alpha+d}}$.
\end{lemma}

\begin{proof}
Let  $\mathcal{N}:=\{j\in\mathbb{N}^d:n\mu_j\geq\sigma^2\}=\{j\in\mathbb{N}^d:\prod_{i=1}^dj_i\leq CJ_{\alpha}\}$ and we apply Lemma \ref{lem:card} [with $\mathcal{I}$=$\mathcal{N}$, $I$=$CJ_{\alpha}$ and $\gamma=k(2\alpha/d+1)-1$].
First, we prove the upper bound,
\begin{align*}
\sum_{j\in\mathbb{N}^d}\nu_j^k&=\sum_{j\in\mathbb{N}^d}\frac{(N\mu_j)^k}{(\sigma^2+N\mu_j)^k}\\
&\leq\sum_{j\in\mathcal{N}}1+\Big(\frac{N}{\sigma^2}\Big)^k\sum_{j\notin\mathcal{N}}\mu_j^k\\
&\lesssim J_{\alpha}(\log J_{\alpha})^{d-1}+\Big(\frac{N}{\sigma^2}\Big)^kJ^{-k(2\alpha/d+1)+1}_{\alpha}(\log J_{\alpha})^{d-1}\\
&\lesssim J_{\alpha}(\log J_{\alpha})^{d-1}.
\end{align*}
The lower bound follows similarly,
\begin{align*}
\sum_{j\in\mathbb{N}^d}\nu_j^k&\geq  \Big(\frac{N}{2\sigma^2}\Big)^k\sum_{j\notin\mathcal{N}}\mu_j^k\gtrsim \Big(\frac{N}{\sigma^2}\Big)^kJ^{-k(2\alpha/d+1)+1}_{\alpha}(\log J_{\alpha})^{d-1}\gtrsim  J_{\alpha}(\log J_{\alpha})^{d-1}.
\end{align*}
\end{proof}

\begin{lemma}\label{lem: bound eigen expo}
	For $\nu_j$, $j\in\mathbb{N}^d$, defined in \eqref{eq: eigenvalues equII} with eigenvalues $\mu_j$ exponentially decaying according to Assumption \ref{ass: bound eigenvalues} with $b=1$, $a<1$ and $k\in\mathbb{N}$,
	\begin{align*}
	\sum_{j\in\mathbb{N}^d}\nu_j^k\asymp J^d_a,
	\end{align*}
	where $J_a=a^{-1}\log(N/\sigma^2)$.
\end{lemma}

\begin{proof}
	Let $\mathcal{N}_d:=\{j\in\mathbb{N}^d:N\mu_j\geq\sigma^2\}=\{j\in\mathbb{N}^d:\sum_{i=1}^dj_i\leq J_a + c/a\}$ with $c>0$ a positive constant. Then it is easy to see that $|\mathcal{N}_d|\leq 2^dJ_a^d$. Moreover, we will show by induction on $d$ that
	$$\sum_{j\notin\mathcal{N}_d}e^{-ak\sum_{i=1}^dj_i}\lesssim a^{-d}\Big(\frac{N}{\sigma^2}\Big)^{-k}\log^{d-1}(N/\sigma^2).$$
	Let us start with the case $d=1$. We can directly see that
	$$\sum_{j>J_a}e^{-akj}\leq Ce^{-akJ_a}\frac{e^{ak}}{e^{ak}-1}\lesssim a^{-1}\Big(\frac{N}{\sigma^2}\Big)^{-k}.$$
	Now, assume that our assumption holds for $d$ and consider the case $d+1$, then
	\begin{align*}
	\sum_{j\notin\mathcal{N}_{d+1}}e^{-ak\sum_{i=1}^{d+1}j_i}&\lesssim \sum_{j_{1:d}\in\mathbb{N}^d}e^{-ak\sum_{i=1}^dj_i}\sum_{j_{d+1}>\max(J_a-\sum_{i=1}^dj_i,0)}e^{-akj_{d+1}}\\
	&\lesssim \sum_{j_{1:d}\in\mathbb{N}^d}(e^{-ak\sum_{i=1}^dj_i}\wedge e^{-akJ_a})\frac{e^{ak}}{e^{ak}-1}\\
	&\lesssim \sum_{j_{1:d}\in\mathcal{N}_d}a^{-1}e^{-akJ_a}+\sum_{j_{1:d}\notin\mathcal{N}_d}a^{-1}e^{-ak\sum_{i=1}^dj_i}\\
	&\lesssim a^{-1}|\mathcal{N}_d|\Big(\frac{N}{\sigma^2}\Big)^{-k}+a^{-d-1}\Big(\frac{N}{\sigma^2}\Big)^{-k}\log^{d-1}(N/\sigma^2)\\
	&\lesssim a^{-d-1}\Big(\frac{N}{\sigma^2}\Big)^{-k}\log^d(N/\sigma^2),
	\end{align*}
	which concludes the induction proof.
	
	Using these two results, we can easily show that
$$\sum_{j\in\mathbb{N}^d}\nu_j^k\lesssim\sum_{j\in\mathcal{N}_d}1+\Big(\frac{N}{\sigma^2}\Big)^k\sum_{j\notin\mathcal{N}_d}e^{-ak\sum_{i=1}^dj_i}\lesssim|\mathcal{N}_d|+a^{-d}\log^{d-1}(n/\sigma^2)\lesssim J^d_a.$$
	
	On the other hand, we can show by induction that for all $J>d$, the cardinality of $\mathcal{N}_d:=\{j\in\mathbb{N}^d:\sum_{i=1}^dj_i\leq J\}$ is bounded from below as follows
	$$|\mathcal{N}_d|\geq(J-d)^d/d!.$$
	Note that it holds trivially for $d=1$. Now assume it holds for $d$, then we can write $\mathcal{N}_{d+1}$ as a partition as follows
	$$\mathcal{N}_{d+1}=\{j\in\mathbb{N}^{d+1}:\sum_{k=1}^{d+1}j_k\leq J\}=\bigcup_{i=1}^{J-d}\{j\in\mathbb{N}^{d+1}:j_{d+1}=i; \sum_{k=1}^dj_k\leq J-i\}.$$
	According to our induction assumption, the cardinality of all these subsets are bounded from below by $(J-d-i)^d/d!$, hence we have
	$$|\mathcal{N}_{d+1}|\geq\sum_{i=1}^{J-d}\frac{(J-d-i)^d}{d!}\geq \int_{1}^{J-d}\frac{(J-d-t)^d}{d!}dt = \frac{(J-d-1)^{d+1}}{(d+1)!},$$
	which concludes our induction proof. Using this result, we can now show that
	$$\sum_{j\in\mathbb{N}^d}\nu_j^k\geq\sum_{j\in\mathcal{N}_d}1=|\mathcal{N}_d|\gtrsim J^d_a,$$
	concluding the proof.
\end{proof}

\begin{lemma}\label{lem:largeJ}
For arbitrary $f_0\in\ell_2(L)$ we get that
\begin{align*}
E_{0}\|\Delta\hat{f}_{n}^{(k)} \|_{\mathcal{H}}^2\leq C N,
\end{align*}
for some universal constant $C>0$.
\end{lemma}

\begin{proof}
	First note that
	\begin{align*}
	\|\Delta \hat{f}_{n}^{(k)} \|_{\mathcal{H}}^2\leq 2\|\hat{f}_{n}^{(k)} \|_{\mathcal{H}}^2+2\|\tilde{F}(f_{0})\|_{\mathcal{H}}^2.
	\end{align*}
	For $f_0\in\ell_2(L)$ the second term on the right hand is bounded by
	\begin{align*}
	\|\tilde{F}(f_{0})\|_{\mathcal{H}}^2=\sum_{j\in \mathbb{N}^d}\mu_j^{-1}\nu_j^2 f_{0,j}^2
	\leq \sum_{j\in \mathbb{N}^d}\frac{N^2\mu_j}{(\sigma^2+\mu_j N)^2}f_{0,j}^2\leq N L^2/\sigma^2 .
	\end{align*}
	
	Then by the definition of $\hat{f}_n^{(k)}$ we get that
	\begin{align}
	\sigma^2 \|\hat{f}_n^{(k)} \|_{\mathcal{H}}^2&\leq \sum_{i=1}^n (\hat{f}_n^{(k)}(X^{(k)}_i)-Y^{(k)}_i)^2+\sigma^2 \|\hat{f}_n^{(k)}\|_{\mathcal{H}}^2\nonumber\\
	&\leq  \Big( \sum_{i=1}^n (\tilde{F}(f_0)(X_i^{(k)})-f_0(X_i^{(k)})-\eps^{(k)}_i)^2+\sigma^2 \|\tilde{F}(f_0)\|_{\mathcal{H}}^2\Big)\nonumber\\
	&\leq2\sum_{i=1}^n \tilde{P}(f_0)^2(X_i^{(k)})+2\sum_{i=1}^n(\eps^{(k)}_i)^2+\sigma^2\|\tilde{F}(f_0)\|_{\mathcal{H}}^2.\label{eq:UB:estimator}
	\end{align}
	We conclude the proof by taking the expectation of both sides
	\begin{align*}
	\sigma^2 E_0\|\hat{f}_n^{(k)} \|_{\mathcal{H}}^2\lesssim \sum_{i=1}^n E_0\tilde{P}(f_0)^2(X_i^{(k)})+\sum_{i=1}^n(\eps^{(k)}_i)^2+\sigma^2\|\tilde{F}(f_0)\|_{\mathcal{H}}^2=O(N).
	\end{align*}
\end{proof}

\begin{lemma}\label{lem:card}
The cardinality of the set 
\begin{align}
\mathcal{I}_{I,d}=\{ j=(j_1,...,j_d)\in\mathbb{N}^d_+:\, \prod_{i=1}^dj_i\leq I \}\label{def:small_index}
\end{align}
 satisfies that $|\mathcal{I}_{I,d}|\leq 2^d I\log^{d-1} I$. Furthermore,
\begin{align}
 \sum_{j\in \mathcal{I}_{I,d}^c} \prod_{i=1}^d j_i^{-\gamma-1}\asymp I^{-\gamma} (I/\log I)^{d-1},
\end{align}
for some universal constants depending only on $\gamma$ and $d$.
\end{lemma}
\begin{proof}
We prove both statement by induction, starting with the first one. For $d=1$ it is trivial. Let us assume that it holds for $d$ and consider the case $d+1$. We distinguish cases according the value of $j_{d+1}$. If $j_{d+1}=1$, then $\prod_{i=1}^d j_i\leq I$ holds, if $j_{d+1}=2$, then $\prod_{i=1}^d j_i\leq I/2$ holds, and so on. Hence we can write that
\begin{align*}
|\mathcal{I}_{I,d+1}|\leq \sum_{j_{d+1}=1}^{I}|\mathcal{I}_{I/j_{d+1},d}|\leq 2^{d}\sum_{j_{d+1}=1}^{I}\frac{I}{j_{d+1}}\log^{d-1} \frac{I}{j_{d+1}}< 2^{d+1} I\log^d I,
\end{align*}
where in the last inequality we have used that $\sum_{i=1}^n1/i<1+\log n < 2\log n$.

Note again that for $d=1$ the second statement holds trivially (using Riemann sums for instance). Then assume that it holds for $d$ and consider the case $d+1$. First we deal with the upper bound, where we note that
\begin{align*}
\sum_{j\in \mathcal{I}_{I,d+1}^c} \prod_{i=1}^{d+1} j_i^{-\gamma-1}&=\sum_{j_{d+1}=1}^{I} j_{d+1}^{-\gamma-1} \sum_{j\in \mathcal{I}_{I/j_{d+1},d}^c}\prod_{i=1}^{d} j_i^{-\gamma-1}\\
&\qquad +  \sum_{j_{d+1}=I}^{\infty}j_{d+1}^{-\gamma-1} \prod_{i=1}^d\sum_{j_i=1}^{\infty} j_i^{-\gamma-1}\\
&\lesssim \sum_{j_{d+1}=1}^{I}j^{-1}_{d+1}I^{-\gamma}\big( \log (I/j_{d+1})\big)^{d-1}+ \sum_{j_{d+1}=I}^{\infty}j_{d+1}^{-\gamma-1}\\
&\leq I^{-\gamma}(\log I)^{d-1}  \sum_{j_{d+1}=1}^{I} j_{d+1}^{-1} + I^{-\gamma}\leq I^{-\gamma}\log^{d}(I). 
\end{align*}
Finally, it remained to deal with the lower bound. First note that it is sufficient to show the result for $I\geq C$, for some $C$ large enough (depending only on $d,\gamma$). Then by noting that for $x\geq e^{d-1}$ the function $x^{-1}\log^{d-1} x$ is monotone decreasing, we get that
\begin{align*}
\sum_{j\in \mathcal{I}_{I,d+1}^c} \prod_{i=1}^{d+1} j_i^{-\gamma-1}&\geq\sum_{j_{d+1}=1}^{I} j_{d+1}^{-\gamma-1} \sum_{j\in \mathcal{I}_{I/j_{d+1},d}^c}\prod_{i=1}^{d} j_i^{-\gamma-1}\\
&\gtrsim I^{-\gamma}\Big( \sum_{j_{d+1}=1}^{I}j^{-1}_{d+1}( \log I)^{d-1}- \sum_{j_{d+1}=1}^{I}j^{-1}_{d+1}( \log j_{d+1})^{d-1}\Big)\\
&\geq I^{-\gamma}\Big(( \log I)^{d-1} \int_{x=1}^{I}x^{-1}dx - \sum_{j_{d+1}=1}^{e^{d-1}}j^{-1}_{d+1}( \log j_{d+1})^{d-1}\\
&\qquad\quad -\int_{x=e^{d-1}}^{I}x^{-1}( \log x)^{d-1}dx\Big)\\
&\geq  I^{-\gamma}\Big( ( \log I)^{d}-C_{d,\gamma}-( \log I)^{d}/2)\Big)\gtrsim I^{-\gamma}( \log I)^{d},
\end{align*}
concluding the proof of our statement.
\end{proof}

\begin{lemma}\label{lem:help:bounded}
	There exists an event $A_n^{(k)}$ such that for any $\theta_0\in L_{\infty}(L)$ and $N\leq n^{C_1}$, for some $C_1\geq 1$ there exist constants $C_2,C_3>0$ such that
	\begin{align*}
	&\left\|\Delta\hat{f}_{n}^{(k)}\right\|_{2} 1_{A_n^{(k)}}\leq n^{C_2},\\
	&E_{f_0}\left\|\Delta\hat{f}_{n}^{(k)} - \tilde{F} \circ F^{-1}\circ \hat{S}^{(k)}_{n}(\tilde{F}(f_0))\right\|_2^21_{(A_n^{(k)})^c}\leq e^{-C_3n}.
	\end{align*}
\end{lemma}

\begin{proof}
	Let us take $A_n^{(k)}=\{\sum_{i=1}^n (\eps_i^{(k)})^2\leq n^{C_0}\}$, for arbitrary $C_0>1$. Then in view of \eqref{eq:UB:estimator} we have on the event $A_n^{(k)}$ that
	\begin{align*}
	\|\Delta\hat{f}_{n}^{(k)}\|_{2}\leq \| \hat{f}_{n}^{(k)}\|_2+ \|\tilde{F}(f_0)\|_2\lesssim N^{1/2}+n^{C_0}+L\lesssim n^{C_0\vee C_1/2}.
	\end{align*}
	Furthermore, note that
	\begin{align*}
	\|\Delta\hat{f}_{n}^{(k)} - \tilde{F} \circ F^{-1}\circ \hat{S}^{(k)}_{n}(\tilde{F}(f_0))\|_2^2&\lesssim \|\Delta\hat{f}_{n}^{(k)}\|_2^2 + \|\tilde{F} \circ F^{-1}\circ \hat{S}^{(k)}_{n}(\tilde{F}(f_0))\|_2^2\\
	&\lesssim N+\sum_{i=1}^n (\eps_i^{(k)})^2+ N^2\|\hat{S}^{(k)}_{n}(\tilde{F}(f_0))\|_2^2.
	\end{align*}
	
	Furthermore from the definition of $\hat{S}^{(k)}_{n}$, the boundedness of $\mathcal{X}$ and $\|K\|_{\infty}=O(1)$ we get that
	\begin{align*}
	\left\|\hat{S}^{(k)}_{n}(\tilde{F}(f_0))\right\|_2^2\lesssim \left\|\hat{S}^{(k)}_{n}(\tilde{F}(f_0))\right\|_\infty^2&\lesssim \left(\frac{1}{n}\sum_{i=1}^n |\eps_i^{(k)}|\right)^2+\|f_0\|_{\infty}^2\\&\lesssim \frac{1}{n}\sum_{i=1}^n(\eps_i^{(k)})^2 +1.
	\end{align*}
	Finally, since $W_n=\sum_{i=1}^n(\eps_i^{(k)})^2\sim \chi_n^2$, note that for $n$ large enough
$$E1_{W_n\geq n^{C_0}}W_n= \int_{\{x\geq n^{C_0}\}} \frac{1}{2^{n/2} \Gamma(n/2)}x^{n/2}e^{-x/2}dx\leq e^{-n^{C_0}/3},$$
reulting in
$$E_{f_0}\left\|\Delta\hat{f}_{n}^{(k)} - \tilde{F} \circ F^{-1}\circ \hat{S}^{(k)}_{n}(\tilde{F}(f_0))\right\|_2^21_{(A_n^{(k)})^c}\lesssim
e^{-n^{C_0}/3}N^2\leq e^{-n},$$
for large enough $n$, concluding the proof of the lemma.

\end{proof}

\begin{lemma}\label{lem: bound-maxfunc}
Let $r,s>0$ such that $r>s/d$ and $f:[1,\infty)^d\to\mathbb{R}$ defined as
$$f(x)=\left(\prod_{i=1}^d x_i\right)^r\left(\sum_{i=1}^dx_i\right)^{-s}.$$
Then $f$ is bounded from above by $d^{-s}J^{r-s/d}$ on the set $\mathcal{N}:=\{x\in[1,\infty)^d:\prod_{i=1}^dx_i\leq J\}$ with $J>1$. 
\end{lemma}

\begin{proof}
From the inequality of arithmetic and geometric means, we know that for all $x\in[1,\infty)^d$
$$\sum_{i=1}^dx_i\geq d\left(\prod_{i=1}^d x_i\right)^{1/d}.$$
Thus, we can bound $f$ from above by
$$f(x)\leq d^{-s}\left(\prod_{i=1}^d x_i\right)^{r-s/d}\leq d^{-s}J^{r-s/d},$$
on $\mathcal{N}$ concluding the proof.
\end{proof}

\begin{lemma}\label{lem: bound-minfunc}
Let $s>0$ and $f:[1,\infty)^d\to\mathbb{R}$ defined as
$$f(x)=(\sum_{i=1}^dx_i)^{-s}.$$
Then $f$ is bounded from above by $d^{-s}J^{-s/d}$ on the set $\mathcal{N}:=\{x\in[1,\infty)^d:\prod_{i=1}^dx_i\geq J\}$ with $J>1$. 
\end{lemma}

\begin{proof}
Since $f$ is differentiable on its domain, we can compute its gradient
$$(\nabla f)_\ell=-s(\sum_{i=1}^dx_i)^{-s-1}<0,$$
for all $\ell\in\{1,...,d\}$. Thus, the function attains its maximum at $\prod_{i=1}^dx_i=J$. At the maximum point, in view the inequality of arithmetic and geometric means, $\sum_{i=1}^dx_i\geq d\left(\prod_{i=1}^d x_i\right)^{1/d}=dJ^{1/d}$. The statement of the lemma follows by raising both sides to the $-s$ power.
\end{proof}



%

\bibliographystyle{acm}
\bibliography{biblio}

\begin{thebibliography}{10}

\bibitem{banerjee:gelfand:ofinley:sang:2008}
{\sc Banerjee, S., Gelfand, A., O'Finley, A., and Sang, H.}
\newblock Gaussian predictive process models for large spatial data sets.
\newblock {\em Journal of the Royal Statistical Society: Series B (Statistical
  Methodology) 70}, 4 (2008), 825--848.

\bibitem{belitser:2014}
{\sc Belitser, E.}
\newblock On coverage and local radial rates of credible sets.
\newblock {\em Ann. Statist. 45}, 3 (06 2017), 1124--1151.

\bibitem{benyi:oh:13}
{\sc B\'{e}nyi, A., and Oh, T.}
\newblock The sobolev inequality on the torus revisited.
\newblock {\em Publicationes Mathematicae Debrecen 83\/} (2013), 359--374.

\bibitem{berlinet:ta:2004}
{\sc Berlinet, A., and C.~{Thomas-Agnan}, C.}
\newblock {\em {RKHS} and Stochastic Processes}.
\newblock Springer US, 2004, pp.~55--108.

\bibitem{bhattacharya:2014}
{\sc Bhattacharya, A., Pati, D., and Dunson, D.}
\newblock Anisotropic function estimation using multi-bandwidth gaussian
  processes.
\newblock {\em Annals of statistics 42}, 1 (2014), 352.

\bibitem{bhattacharya:pati:yun:17}
{\sc Bhattacharya, A., Pati, D., and Yang, Y.}
\newblock Frequentist coverage and sup-norm convergence rate in gaussian
  process regression.
\newblock {\em arXiv e-prints\/} (2017).

\bibitem{burt2019}
{\sc Burt, D.~R., Rasmussen, C.~E., and van~der Wilk, M.}
\newblock Rates of convergence for sparse variational {G}aussian process
  regression.
\newblock In {\em International Conference on Machine Learning\/} (2019), PMLR,
  pp.~862--871.

\bibitem{castillo:nickl:2013}
{\sc Castillo, I., and Nickl, R.}
\newblock Nonparametric {B}ernsteinñ-von {M}ises theorems in gaussian white
  noise.
\newblock {\em Ann. Statist. 41}, 4 (08 2013), 1999--2028.

\bibitem{cobos:kuhn:sickel:15}
{\sc Cobos, F., K\"{u}hn, T., and Sickel, W.}
\newblock Optimal approximation of multivariate periodic {S}obolev functions in
  the sup-norm.
\newblock {\em Journal of Functional Analysis 270\/} (2015).

\bibitem{cressie:2015}
{\sc Cressie, N.}
\newblock {\em Statistics fo spatial data}.
\newblock John Wiley \& Sons, 2015.

\bibitem{currin:mitchell:morris:ylvisaker:1991}
{\sc Currin, C., Mitchell, M., Morris, M., and Ylvisaker, D.}
\newblock Bayesian prediction of deterministic functions, with applications to
  the design and analysis of computer experiments.
\newblock {\em Journal of the American Statistical Association 86}, 416 (1991),
  953--963.

\bibitem{deisenroth:2015}
{\sc Deisenroth, M., and Ng, J.~W.}
\newblock Distributed {G}aussian processes.
\newblock In {\em Proceedings of the 32nd International Conference on Machine
  Learning\/} (Lille, France, 07--09 Jul 2015), F.~Bach and D.~Blei, Eds.,
  vol.~37 of {\em Proceedings of Machine Learning Research}, PMLR,
  pp.~1481--1490.

\bibitem{opper:trecate:williams:98}
{\sc Ferrari-Trecate, G., Williams, C., and Opper, M.}
\newblock Finite-dimensional approximation of gaussian processes.
\newblock In {\em Advances in Neural Information Processing Systems\/} (1998),
  M.~Kearns, S.~Solla, and D.~Cohn, Eds., vol.~11, MIT Press.

\bibitem{ghosal2017fundamentals}
{\sc Ghosal, S., and Van~der Vaart, A.}
\newblock {\em Fundamentals of nonparametric Bayesian inference}, vol.~44.
\newblock Cambridge University Press, 2017.

\bibitem{gibbs:poole:stockmeyer:1976}
{\sc Gibbs, N., Jr, W.~P., and Stockmeyer, P.}
\newblock An algorithm for reducing the bandwidth and profile of a sparse
  matrix.
\newblock {\em SIAM J. Numer. Anal. 13}, 2 (1976), 236--250.

\bibitem{guhaniyogi2017divide}
{\sc Guhaniyogi, R., Li, C., Savitsky, T.~D., and Srivastava, S.}
\newblock A divide-and-conquer bayesian approach to large-scale kriging.
\newblock {\em arXiv preprint arXiv:1712.09767\/} (2017).

\bibitem{hadji:szabo:2021}
{\sc Hadji, A., and Szabo, B.}
\newblock Can we trust bayesian uncertainty quantification from gaussian
  process priors with squared exponential covariance kernel?
\newblock {\em SIAM/ASA Journal on Uncertainty Quantification 9}, 1 (2021),
  185--230.

\bibitem{hunter:13}
{\sc Hunter, J.}
\newblock Distributions and sobolev spaces.
\newblock Lecture Notes: Analysis Prelim Workshop, 2013.
\newblock Department of Mathematics of the University of California Davis.

\bibitem{isserlis:1916}
{\sc Isserlis, L.}
\newblock On certain probable errors and correlation coefficients of multiple
  frequency distributions with skew regression.
\newblock {\em Biometrika 11}, 3 (1916), 185--190.

\bibitem{kennedy:ohagan:2001}
{\sc Kennedy, M., and O'Hagan, A.}
\newblock Bayesian calibration of computer models.
\newblock {\em Journalof the Royal Statistical Society: Series B (Statistical
  Methodology) 63}, 3 (2001), 425--464.

\bibitem{knapik}
{\sc Knapik, B., van~der Vaart, A.~W., and van Zanten, J.~H.}
\newblock Bayesian inverse problems with gaussian priors.
\newblock {\em Ann. Statist. 39}, 5 (2011), 2626--2657.

\bibitem{kocijan2016modelling}
{\sc Kocijan, J.}
\newblock {\em Modelling and control of dynamic systems using Gaussian process
  models}.
\newblock Springer, 2016.

\bibitem{mallasto:feragen:17}
{\sc Mallasto, A., and Feragen, A.}
\newblock Learning from uncertain curves: The 2-{W}asserstein metric for
  {G}aussian processes.
\newblock In {\em Advances in Neural Information Processing Systems 30},
  I.~Guyon, U.~V. Luxburg, S.~Bengio, H.~Wallach, R.~Fergus, S.~Vishwanathan,
  and R.~Garnett, Eds. Curran Associates, Inc., 2017, pp.~5660--5670.

\bibitem{matheron:1973}
{\sc Matheron, G.}
\newblock The intrinsic random functions and their applications.
\newblock {\em Advances inapplied probability 5}, 3 (1973), 439--468.

\bibitem{mitchell:sacks:welch:wynn:1989}
{\sc Mitchell, T., Sacks, J., Welch, W., and Wynn, H.}
\newblock Design and analysis of computer experiments.
\newblock {\em Statistical science\/} (1989), 409--423.

\bibitem{nieman2021contraction}
{\sc Nieman, D., Szabo, B., and van Zanten, H.}
\newblock Contraction rates for sparse variational approximations in gaussian
  process regression.
\newblock {\em arXiv preprint arXiv:2109.10755\/} (2021).

\bibitem{opper:vivarelli:1999}
{\sc Opper, M., and Vivarelli, F.}
\newblock General bounds on bayes errors for regression with gaussian
  processes.
\newblock In {\em Advances in Neural Information Processing Systems II},
  M.~Kearns, S.~Solla, and D.~Cohn, Eds. MIT Press, 1999, pp.~302--308.

\bibitem{quinonero:rasmussen:2005}
{\sc Qui{\~n}onero-Candela, J., and Rasmussen, C.~E.}
\newblock A unifying view of sparse approximate gaussian process regression.
\newblock {\em J. Machine Learning Research 6\/} (2005), 1939--1959.

\bibitem{rasmussen:williams:2006}
{\sc Rasmussen, C., and Williams, C.}
\newblock {\em Gaussian processes for machine learning}.
\newblock MIT Press, Boston, 2006.

\bibitem{ray2017}
{\sc Ray, K.}
\newblock Adaptive bernstein-von mises theorems in gaussian white noise.
\newblock {\em Ann. Statist. 45}, 6 (12 2017), 2511--2536.

\bibitem{rousseau:szabo:2015:main}
{\sc {Rousseau}, J., and {Szabo}, B.}
\newblock Asymptotic behaviour of the empirical bayes posteriors associated to
  maximum marginal likelihood estimator.
\newblock {\em Ann. Statist. 45\/} (2017), 833--865.

\bibitem{rousseau:szabo:16:main}
{\sc Rousseau, J., and Szabo, B.}
\newblock {Asymptotic frequentist coverage properties of Bayesian credible sets
  for sieve priors}.
\newblock {\em The Annals of Statistics 48}, 4 (2020), 2155 -- 2179.

\bibitem{saad:1990}
{\sc Saad, Y.}
\newblock Sparskit: a basic tool kit for sparse matrix computations, 1990.

\bibitem{scott:blocker:et.al:16}
{\sc Scott, S., Blocker, A., Bonassi, F., Chipman, H., George, E., and
  McCulloch, R.}
\newblock Bayes and big data: {T}he consensus monte carlo algorithm.
\newblock {\em International Journal of Management Science and Engineering
  Management 11}, 2 (2016), 78--88.

\bibitem{sniekers:2015}
{\sc Sniekers, S., and van~der Vaart, A.}
\newblock Adaptive {B}ayesian credible sets in regression with a {G}aussian
  process prior.
\newblock {\em Electron. J. Stat. 9}, 2 (2015), 2475--2527.

\bibitem{srivastava:2015a}
{\sc Srivastava, S., Cevher, V., Dinh, Q., and Dunson, D.}
\newblock {WASP: Scalable {B}ayes via barycenters of subset posteriors}.
\newblock In {\em Proceedings of the Eighteenth International Conference on
  Artificial Intelligence and Statistics\/} (San Diego, California, USA, 09--12
  May 2015), G.~Lebanon and S.~V.~N. Vishwanathan, Eds., vol.~38 of {\em
  Proceedings of Machine Learning Research}, PMLR, pp.~912--920.

\bibitem{stein:2012}
{\sc Stein, M.}
\newblock {\em Interpolation of Spatial Data: some theory for kriging}.
\newblock Springer Science \& Business Media. 2012.

\bibitem{szabo:vzanten:19}
{\sc Szab{{\'o}}, B., and van Zanten, H.}
\newblock An asymptotic analysis of distributed nonparametric methods.
\newblock {\em Journal of Machine Learning Research 20}, 87 (2019), 1--30.

\bibitem{szabo:vdv:vzanten:13}
{\sc Szabo, B.~T., van~der Vaart, A.~W., and van Zanten, J.~H.}
\newblock {Frequentist coverage of adaptive nonparametric Bayesian credible
  sets}.
\newblock {\em Annals of Statistics 43}, 4 (2015), 1391--1428.

\bibitem{titsias:2009}
{\sc Titsias, M.}
\newblock Variational learning of inducing variables in sparse {G}aussian
  {P}rocesses.
\newblock In {\em Artificial Intelligence and Statistics}. 2009, pp.~567--574.

\bibitem{tresp2000bayesian}
{\sc Tresp, V.}
\newblock A bayesian committee machine.
\newblock {\em Neural computation 12}, 11 (2000), 2719--2741.

\bibitem{vzanten:vdv:07}
{\sc van~der Vaart, A., and van Zanten, J.~H.}
\newblock Bayesian inference with rescaled {G}aussian process priors.
\newblock {\em Electron. J. Statist. 1\/} (2007), 433--448.

\bibitem{vzanten:vdv:11}
{\sc van~der Vaart, A., and van Zanten, J.~H.}
\newblock Information rates of nonparametric {G}aussian process methods.
\newblock {\em Journal of Machine Learning Research 12\/} (2011), 2095--2119.

\bibitem{vvvz08}
{\sc van~der Vaart, A.~W., and van Zanten, J.~H.}
\newblock Rates of contraction of posterior distributions based on {G}aussian
  process priors.
\newblock {\em Ann. Statist. 36}, 3 (2008), 1435--1463.

\bibitem{yang:2016}
{\sc Yang, Y., and Dunson, D.~B.}
\newblock Bayesian manifold regression.
\newblock {\em Ann. Statist. 44}, 2 (04 2016), 876--905.

\bibitem{yoo:2016}
{\sc Yoo, W.~W., and Ghosal, S.}
\newblock Supremum norm posterior contraction and credible sets for
  nonparametric multivariate regression.
\newblock {\em Ann. Statist. 44}, 3 (06 2016), 1069--1102.

\end{thebibliography}

\end{document}